\documentclass{amsart}

\usepackage{graphicx}
\usepackage{amsfonts}
\usepackage{amsmath}
\usepackage{amsthm}

\newtheorem{Thm}{Theorem}

\newtheorem{Cor}{Corollary}
\newtheorem{Lem}{Lemma}
\newtheorem{Prop}{Proposition}
\newtheorem{Ques}{Question}

\newtheorem{Rem}{Remark}

\newcommand{\q}{\mathbb{Q}}

\newcommand{\n}{\mathbb{N}}
\newcommand{\z}{\mathbb{Z}}
\newcommand{\f}{\mathbb{F}}

\begin{document}
\title[Elliptic Curves and Chip-Firing Games]{Combinatorial Aspects of Elliptic Curves II: \\
Relationship between Elliptic Curves and Chip-Firing Games on Graphs}
\author{Gregg Musiker}
\thanks{This work was partially supported by the NSF, grant DMS-0500557
during the author's graduate school at the University of California, San Diego.
}
\address{Mathematics Department, Massachusetts Institute of Technology, 
Cambridge, MA 02139} \email{musiker@math.mit.edu}

\date{September 29, 2007} 

\maketitle

\tableofcontents

\begin{abstract}
Let $q$ be a power of a prime and $E$ be an elliptic curve defined over $\f_q$. In \cite{Mus}, the present author examined a sequence of polynomials which express the $N_k$'s, the number of points on $E$ over the field extensions $\f_{q^k}$, in terms of the parameters $q$ and $N_1 = \#E(\f_q)$. These polynomials have integral coefficients which alternate in sign, and a combinatorial interpretation in terms of spanning trees of wheel graphs.  In this sequel, we explore further ramifications of this connection.  In particular, we highlight a relationship between elliptic curves and chip-firing games on graphs by comparing the groups structures of both.  As a coda, we construct a cyclic rational language whose zeta function is dual to that of an elliptic curve.
\end{abstract}

\section{Introduction}

The theory of elliptic curves is quite rich, arising in both complex
analysis and number theory.  In particular, they can be given a
group structure using the tangent-chord method or the divisor class
group of algebraic geometry \cite{Silver}.
This property makes them not only geometric but also algebraic
objects and allows them to be used for cryptographic purposes
\cite{Wash}.

In \cite{Mus}, the author started an exploration of elliptic curves
from a combinatorial viewpoint.  For a given elliptic curve $E$
defined over a finite field $\f_q$, we let $N_k = \#E(\f_{q^k})$
where $\f_{q^k}$ is a $k$th degree extension of the finite field
$\f_q$. Because the zeta function for $E$, i.e.
$$ \exp\bigg(\sum_{k\geq 1} \frac{N_k }{ k}T^k \bigg) = \frac{1- (1+q-N_1)T + qT^2 }{
(1-T)(1-qT)},$$ only depends on $q$ and $N_1$, the sequence
$\{N_k\}$ only depends on those two parameters as well. More specifically,
Adriano Garsia observed that these bivariate expressions for $N_k$ are in fact
polynomials with integer coefficients, which alternate in sign with
respect to the power of $N_1$ \cite{Gars}.

This motivated the main topic of \cite{Mus}, which was the search
for a combinatorial interpretation of these coefficients.  One such
interpretation discussed therein 
involved a sequence denoted as $\mathcal{W}_k(q,t)$,
a $(q,t)$-deformation of the number of spanning trees of a certain
family of graphs known as the \emph{wheel graphs}.  In this sequel, we more deeply explore this combinatorial 
interpretation.  In particular, the number of spanning trees of a graph, also 
known as the graph's complexity, is an important characteristic of a graph, 
used to study connectivity, with applications to networks.  Additionally, this 
quantity is known to enumerate other structures such as the order of a graph's 
critical group, and as more recently observed in \cite{PostSh}, the number of $G$-parking functions associated to graph $G$.  Here we investigate the connection to critical groups for wheel graphs.  We describe several properties that these critical groups share with elliptic curve groups, thus demonstrating a relationship between these structures. (See Theorem $4$.)

The outline of this paper will be as follows.  We start by reviewing the definitions and 
two theorems of \cite{Mus}, which are labelled as Theorems \ref{MainMus} and \ref{detformu} below.  In the present paper, we in fact provide an alternate definition of $\mathcal{W}_k(q,t)$ and an alternative proof of Theorem \ref{detformu} which did not appear in \cite{Mus}.  We provide this proof because it will use the same terminology that will appear elsewhere in the paper.  We then switch gears, and in Section $3$, discuss critical groups of graphs and the subject of chip firing games.  We will include background material to try to make this paper self-contained, but most of the details of this section come from Norman Biggs \cite{Biggs}.

In Section $4$, we specialize this theory to the case of a family of graphs that we refer to as $(q,t)$-wheel graphs, and explicitly describe critical configurations for them (Theorems $5$ and $6$).  We return to the topic of elliptic curves over finite fields in 
Section $5$, and more closely study the Frobenius map on these varieties.  This 
section will involve introductory material but we delay its inclusion until this place in the paper since it will not be used in earlier sections.  We also describe elliptic cyclotomic polynomials, which first appeared in \cite{Mus}.  This 
will lead us to Section $6$, where the main result, Theorem \ref{geomWcyc}, involves wheel graph analogues of elliptic curves over the algebraic closure of $\f_q$ and the Frobenius map.

We conclude Section $6$ with several additional applications of this point of view.  These include a characterisitc equation for the wheel graph Frobenius map (Theorem $9$), and explicit group presentations as expressed in Theorem $10$, Corollary $2$, and Theorem $12$.  Additionally we answer a question of Norman Biggs \cite{BiggsPersonal} and, in Theorem $11$, generalize a result of his on the cyclicity of deformed wheel graphs \cite{BiggsCrypto}.  In short, we will see that the critical groups of the $(q,t)$-wheel graphs decompose into at most two cyclic groups, just like elliptic curves over finite fields.

Finally in Section $7$, we come full-ciricle and present the theory of zeta functions again.  However, this time around we shall be considering zeta functions which arise in the theory of combinatorics on words.  In particular, we consider particular subsets of strings arising from a given alphabet, a recognizable language in computer science terminology.  Berstel and Reutenauer 
\cite{BerReut} defined a zeta function for this family of objects and it is 
their definition that we utlize in this section.  In particular they found 
that in the case that a language is cyclic and recognizable by a determinisitc finite automaton, 
then its corresponding zeta function is in fact rational.  We conclude the paper 
by explicitly computing the zeta function for a particular family of cyclic rational languages 
and comparing this with the Hasse-Weil zeta function of an elliptic curve, as given in Theorem \ref{ZetaLangQT}.

\vspace{2em}

{\bf Acknowledgements.}  This work first appeared in the author's Ph.D. Thesis at the University of California, San Diego, alongside the article ``Combinatorial Aspects of Elliptic Curves'' \cite{Mus}.  The author enthusiastically thanks his advisor Adriano Garsia for his guidance and many useful conversations.  Conversations with Norman Biggs, Christophe Reutenauer, Richard Stanley, and Nolan Wallach have also been very helpful.  The author would like to thank the NSF and the ARCS Foundation for their support during the author's graduate school.  A preprint of the present article was presented at FPSAC 2007 in Tianjin, China.

\section{An Enumerative Correspondence between Elliptic Curves and Wheel Graphs}

Let $W_k$ denote the $k$th wheel graph, which consists of $(k+1)$ vertices, $k$ of which lie in a cycle and 
are each adjacent to the last vertex.  (We also define $W_k$ analogously in the case 
$k=1$ or $k=2$, each having a degenerate cycle of length one or two, respectively.)
A spanning tree of a graph is a connected subgraph which does not contain any cycles.  
In the case of the wheel graphs, a spanning tree is easily defined as a collection of 
disconnected arcs on the rim, which each connect to the central hub along one spoke for 
each arc.  In \cite{Mus}, we defined a $(q,t)$-weighting for such spanning trees $T$ by letting the exponent of $t$ be the number of spokes in $T$ and the exponent of $q$ signify the total number of edges lying clockwise with respect to the unqiue spoke associated to that particular arc, which we abbreviate as $dist(T)$.  With this weighting in mind, the main result of \cite{Mus} was the following. 
\begin{Thm} [Theorem $3$ of \cite{Mus}] \label{MainMus}
The number of points on an elliptic curve $E$ over finite field $\f_{q^k}$, which 
we denote as $N_k$, satisfies the identity 
$$N_k = -\mathcal{W}_k(q,-N_1)$$ where $N_1 = \# E(\f_q)$ and 
$\mathcal{W}_k(q,t) = \sum_{T\mathrm{~a~spanning~tree~of~}W_k} ~q^{~dist(T)}~t^{~\#spokes(T)}.$
\end{Thm}
\noindent In this paper, we use a slightly different definition for $\mathcal{W}_k(q,t)$, 
which will allow us to expand our results to other areas of combinatorics.  In particular, instead of simply using the family of wheel graphs, we define a $(q,t)$-deformation of this family where the graphs are no longer simple or undirected.  In other words we use a weighting scheme such that the graphs themselves change rather than the way in which we enumerate $W_k$'s spanning trees.

Define $W_k(q,t)$ to be the following directed graph (digraph) with multiple edges: We use the $0$-skeleton of the wheel graph $W_k$, where we label the central vertex as $v_0$, and the vertices on the rim as $v_1$ through $v_k$ in clockwise order.  We then attach $t$ bi-directed spokes 
between $v_0$ and $v_i$ for all $i\in \{1,2,\dots, k\}$.  Additionally, we attach a single counter-clockwise edge between $v_i$ and $v_{i-1}$ (working modulo $k$) for each vertex on the rim.  Finally, we attach $q$ clockwise edges between $v_i$ and $v_{i+1}$ (again working modulo $k$). 

\begin{Prop}
The number of directed spanning trees of $W_k(q,t)$, rooted at vertex $v_0$ equals the polynomial $\mathcal{W}_k(q,t)$.
\end{Prop}

\begin{proof}
By comparing this new definition with the original one from \cite{Mus}, we simply note that we have translated the above weighting into a scheme where we have multiple edges in $W_k(q,t)$ whenever we have a weight in $\mathcal{W}_k(q,t)$.
\end{proof}

By transitivity we arrive at the fact that the sequence of $\{N_k\}$'s are in fact a signed version of the number of rooted spanning trees in this family of multi-digraphs.
As an immediate application of this different characterization of $\mathcal{W}_k(q,t)$, we obtain another proof of the determinantal formula for $N_k$ which appeared in 
\cite{Mus}.

Define the family of matrices $M_k$ by 
 $M_1 = \left[ -N_1\right]$, $M_2 = \left[\begin{matrix} 1+q-N_1
& -1-q \\ -1-q & 1+q-N_1\end{matrix}\right]$, and for $k\geq 3$, let
$M_k$ be the $k$-by-$k$ ``three-line'' circulant matrix
$$\left[
\begin{matrix}
1+q-N_1 & -q & 0 &\dots & 0 &-1 \\
-1 & 1 +q -N_1 &  -q & 0 & \dots& 0 \\
\dots & \dots & \dots & \dots & \dots & \dots\\
0 & \dots & -1 & 1+q-N_1 & -q & 0 \\
0 & \dots & 0 & -1 & 1+q-N_1 & -q \\
-q & 0 & \dots & 0 & -1 & 1+q-N_1
\end{matrix}\right].$$

\begin{Thm} [Theorem $5$ in \cite{Mus}]
\label{detformu}
The sequence of integers $N_k = \#E(\f_{q^k})$ satisfies the
relation $$N_k = -\det M_k$$ for all $k\geq 1$.  We obtain an
analogous determinantal formula for $\mathcal{W}_k(q,t)$, in fact
$\mathcal{W}_k(q,t) = \det M_k|_{N_1 = -t}$.
\end{Thm}

\begin{proof}
We appeal to the directed multi-graph version of the Matrix-Tree
Theorem \cite{EC2} to count the number of spanning trees of $W_k(q,t)$ with
root given as the hub.  The Laplacian $L$ of a digraph on $m$ vertices, with possibly multiple edges, is defined to be the $m$-by-$m$ matrix in which off-diagonal entries 
$L_{ij} = -d(i,j)$ and diagonal entries $L_{ii}=d(i)$.  Here $d(i,j)$ is the number of edges from $v_i$ to $v_j$, and $d(i)$ is the outdegree of vertex $v_i$, or more simply we choose 
$L_{ii}$ such that each row of $L$ sums to zero. 
In the case of $W_k(q,t)$, the Laplacian matrix is
$$ L = \left[
\begin{matrix}
1+q+t & -q & 0 &\dots & 0 &-1 & -t \\
-1 & 1 +q +t &  -q & 0 & \dots& 0 & -t \\
\dots & \dots & \dots & \dots & \dots & \dots & -t \\
0 & \dots & -1 & 1+q+t & -q & 0 & -t\\
0 & \dots & 0 & -1 & 1+q+t & -q & -t\\
-q & 0 & \dots & 0 & -1 & 1+q+t & - t \\
-t & -t & -t & \dots & -t & -t & k t
\end{matrix}\right]$$
where the last row and column correspond to the hub vertex.  We wish to count the number of directed spanning trees rooted at the hub, and the Matrix-Tree Theorem states that this number is given by $\det L_0$ where $L_0$ is matrix
$L$ with the last row and last column deleted.  From this, we obtain the
identities
\begin{eqnarray*}N_k ~~~&=& -\mathcal{W}_k(q,-N_1) \\
M_k ~~~&=& ~~~~~~~~~~L_0\bigg|_{t=-N_1} \mathrm{~~~and~thus} \\
\mathcal{W}_k(q,t) ~~~~~~~~~&=& ~~~~~\det L_0 \mathrm{~~~~~~~~~~~~implies} \\
-\mathcal{W}_k(q,-N_1) &=& ~~-\det L_0\bigg|_{t=-N_1}
\mathrm{~~~so~we~get} \\
N_k ~~~&=& ~~-\det M_k.
\end{eqnarray*}
Thus we have proven Theorem \ref{detformu}.
\end{proof}

We will return to ramifications of this combinatorial identity in Section $5$, after discussing another instance of the graph Laplacian.

\section{Introduction to Chip-Firing Games}

We step away from elliptic curves momentarily and discuss some
fundamental results from the theory of chip-firing games on graphs as described by Bj\"{o}rner, Lov\'{a}sz, and Shor \cite{BLShor}. These are also known as abelian sandpile groups as described by Dhar \cite{Dhar}.  Gabrielov wrote one of the first papers describing the relationship between these two models \cite{Gab}.  The main source for the details we will use is \cite{Biggs}, though there is an extensive literature on the subject, for example see \cite{chipsummary} for a summary.

At first glance, this topic might appear
totally unrelated to elliptic curves, but we will shortly flesh out
the connection. Given a directed (loop-less) graph $G$, we define a
configuration $C$ to be a vector of nonnegative integers, with a
coordinate for each vertex of the graph, letting $c_i$ denote the
integer corresponding to vertex $v_i$.  One can think of this
assignment as a collection of chips placed on each of the vertices.
We say that a given vertex $v_i$ can \emph{fire} if the number of
chips it holds, $c_i$, is greater than or equal to its out-degree.
If so, firing leads to a new configuration where a chip travels
along each outgoing edge incident to $v_i$.  Thus we obtain a
configuration $C^\prime$ where $c_j^\prime = c_j + d(v_i,v_j)$ and
$c_i^\prime = c_i - d(v_i)$. Here $d(v_i,v_j)$ equals the number of
directed edges from $v_i$ to $v_j$, and $d(v_i)$ is the out-degree
of $v_i$, which of course equals $\sum_{j \not = i} d(v_i, v_j)$.

Many interesting problems arise from this definition.  For example,
it can be shown \cite{Lat} that the set of configurations reachable
from an initial choice of a vector forms a distributive lattice.
Thus one can ask combinatorial questions such as examining the
structure of this lattice as a poset.  Other computations such as
the minimal number or expected number of firings necessary to reach
configuration $C^\prime$ from $C$ are also common in dynamical
systems.

A variant of the standard chip-firing game, known as the
\emph{dollar game}, due to Biggs \cite{Biggs} has the same set-up as before with three changes.
\begin{enumerate}
\item We designate one vertex $v_0$ to be the bank, and allow $c_0$
to be negative.  All the other $c_i$'s still must be nonnegative.

\item To limit extraneous configurations, we presume
that the sum $\sum_{i=0}^{\#V-1} c_i = 0$.  (Thus in particular,
$c_0$ will be non-positive.)

\item The bank, i.e. vertex $v_0$, is only allowed to fire if no
other vertex can fire.  Note that since we now allow $c_0$ to be
negative, $v_0$ is allowed to fire even when it is smaller than its
outdegree.

\end{enumerate}
\noindent A configuration is
\emph{stable} if $v_0$ is the only vertex that can fire, and
configuration $C$ is \emph{recurrent} if there is a firing
sequence which leads back to $C$.  Note that this will necessarily
require the use of $v_0$ firing.  We call a configuration \emph{critical} if it is both stable and recurrent.

\begin{Prop} \label{uniquecrit}
For any initial configuration satisfying rules $(1)$ and $(2)$
above, there exists a \emph{unique} critical configuration that can
be reached by a firing sequence, subject to rule $(3)$.
\end{Prop}

\begin{proof}
See \cite{Gab} for original proof, or \cite{Biggs} for slightly different technique.
\end{proof}

The \emph{critical group of graph} $G$, with respect to
vertex $v_0$ is the set of critical configurations, with addition
given by $C_1 \oplus C_2 = \overline{ C_1 + C_2}$.  Here $+$
signifies the usual pointwise vector addition and $\overline{C_3}$
represents the unique critical configuration reachable from $C_3$.
When $v_0$ is understood, we will abbreviate this group as the
critical group of graph $G$, denoting it as $K(G)$.

\begin{Thm} [Gabrielov \cite{Gab}]
$K(G)$ is in fact an abelian (associative) group.
\end{Thm}

\begin{proof}
If we consider the initial configuration $C_3= C_1+C_2$, then by
Proposition \ref{uniquecrit}, there is a unique critical
configuration reachable from $C_3$.  Additionally, we can compute
$(C_0\oplus C_1)\oplus C_2$ or $C_0 \oplus  (C_1 \oplus C_2)$ by
adding together $C_0+C_1+C_2$ pointwise, and then reducing once at
the end, rather than reducing twice.  Thus associativity and
commutativity follow.
\end{proof}

The savvy reader might have noticed that the firing of vertex $v_i$ alters the 
configuation vector exactly as the subtraction of the $i$th row of the Laplacian matrix.  In fact, for any graph we have the following general fact.
\begin{Prop}
If $K(G)$ denotes the critical group of graph $G$, on ($k+1$) vertices, with bank vertex $v_0$, and $L_0$ denotes the reduced Laplacian of $G$ with the row and column corresponding to $v_0$ deleted, then
$$K(G) \cong \mathrm{~coker~}L_0 = ~~~\z^k \bigg / \mathrm{~Im~}L_0 z^k.$$
\end{Prop}
\begin{Cor} \label{critIstree}
$|K(G)| = \det(L_0) = \#\{\mathrm{directed~rooted~spanning~trees~of~graph~}G\}.$
\end{Cor}

\begin{proof}
We use the algebraic fact that when a matrix $M$ is nonsingular, $|\det(M)|=|\mathrm{coker~}M|$ for the first equality.  The second equality follows from the Matrix-Tree Theorem.
\end{proof}

Corollary \ref{critIstree} allows us to extend the identities of Theorems \ref{MainMus} and \ref{detformu} to one which exhibits a reciprocity between the two families of groups described above.
 
\begin{Thm} \label{GpCorresp}
Letting $N_k(q,N_1)$ be the bivariate expression for the cardinality $|E(\f_{q^k})|$ and $K(W_k(q,t))$ be the crticial group on the $(k+1)$ vertex $(q,t)$-wheel graph, we have
$$\bigg|K\bigg((q,t)\mathrm{-}W_k\bigg)\bigg| =  -N_k(q,-t).$$
\end{Thm}

It is this theorem that motivates the remainder of this paper as we explore deeper properties of the $W_k(q,t)$'s and compare them to the case of elliptic curves.

\section{Critical Configurations for the $\mathcal{W}_k(q,t)$ Graphs}

We begin our exploration by completely characterizing critical configurations
of the $(q,t)$-wheel graphs.  This new characterization of critical configurations also yields a bijection between critical configurations and spanning trees, as given in Theorem \ref{bijthm}.

We take root and hub $v_0$ to be the bank vertex as a convention, and thus a configuration of this graph is a vector of length $k$ which encodes the number of chips on each of the rim vertices, which are labelled in clockwise order.

\begin{Lem} \label{Lemzero} A configuration $C = [c_1,c_2,\dots, c_k]$ of the wheel graph $W_k(q,t)$ is stable if and only if $0 \leq c_i \leq q+t$ for all
$1 \leq i \leq k$.  Furthermore, any configuration which is not stable can be reduced to a stable one by the chip-firing rules.\end{Lem}

\begin{proof} It is clear that we disallow $c_i < 0$ as a legal configuration by our definition.  If such a configuration were to come up, we
could add $t$ to every value $c_i$, simulating the firing of the
central vertex, until we have a nonnegative vector.  If on the other hand, there exists $c_i \geq 1+q+t$, with all other $c_i \geq 0$, then vertex $v_i$ can fire
resulting in a new nonnegative configuration, with the sum of the $c_i$'s having been decreased by $t$.  Thus eventually, we will arrive at a configuration with all $c_i$'s satisfying $0\leq c_i\leq q+t$.  Otherwise, if all
$c_i$ are in the specified range, we have a stable configuration
where no vertex except the hub can fire.
\end{proof}

\begin{Lem} \label{firlem} Let $C=[c_1,\dots c_k]$ be a stable configuration.  Then $C$ is critical if and only if $C+[t] = [c_1+t,\dots c_k+t]$ is not stable.
\end{Lem}

\begin{proof}
The stability of $C$ implies that the hub vertex is the only one that can fire.  Configuration $C+[t]$ is either stable as well, or $C+[t]$ reduces to $C$ via the firing of vertices $v_1$ through $v_k$, each exactly once.  In the case that $C+[t]$ is stable, then there exists some minimum integer $d \geq 2$ such that $C+[dt]$ is not stable, but such a configuration reduces to $C+[(d-1)t]$, and thus $C$ does not recur.
\end{proof}

\begin{Lem} \label{critObb} Any critical configuration $[c_1,\dots, c_k]$ will have at least one element $c_i = B$ such that
$B \in \{1+q,\dots, q+t\}$.
\end{Lem}

\begin{proof}
Assume otherwise.  Then $c_i \in \{0,1,\dots, q\}$ for all $1 \leq i
\leq k$.  Consequently, we may add $t$ to every $c_i$ and still
obtain a stable configuration.  Thus the initial configuration is
not critical by Lemnma \ref{firlem}.
\end{proof}

\begin{Thm} \label{critcircc}
Any configuration $C$ is critical if and only if it consists of a
circular concatenation of blokcs of the form
$$B,M_1,\dots, M_j \mathrm{~~~~~~or~~~~~~} B, M_1, \dots, M_j, 0 \mathrm{~~~~~~or~~~~~~} B, M_1, \dots, M_j, 0, q, q, \dots, q$$ where $B \in \{1+q,\dots, q+t\}$ and $M_i \in \{1,\dots, q\}$.
\end{Thm}

We have already shown that there exists at least one $c_i = B$ with
$B > q$.  Thus we prove this theorem by induction on $n$, the number
of such elements. Consider such a block in context, and presume it
is of the form \begin{eqnarray} \label{blockstar} \cdots, M_n^{k_n}~|~B_1, M_1^1, M_1^2, \dots, M_1^{k_1}~|~ B_2 , \cdots \end{eqnarray} where $M_p^i \in \{0,1,\dots, q\}$ and $B_p \in \{1+q,\dots, q+t\}$.  Here $M_n^{k_n}$ and $B_2$
represent the end of the previous block and the beginning of the
next block, respectively.  The heart of the proof is the
verification of the following proposition.

\begin{Prop}
A configuration in the form of (\ref{blockstar}) cannot be recurrent unless $M_p^{j_p}=0$
implies that the remaining $M_p^i$'s, i.e. $M_p^{j_p+1}$ through
$M_p^{k_p}$, are equal to $q$.
\end{Prop}

\begin{proof}
Without loss of generality, we will work with $p=1$ and let $j_1=j$,
$k_1=k$, $M_n^{k_n}=M_0$.  Assume that $M_1^1$ through $M_1^{j-1}
\in \{1,2,\dots q\}$.  We add $t$ to every element of $C$, getting
$C+[t]$, and then reduce via the chip-firing rules whenever we
encounter an element with value greater or equal to $1+q+t$.
Configuration $C+[t]$ contains element $B_1+t$, with value $\geq
1+q+t$, but all other elements of the block are $< 1+q+t$.  Once we
replace $B_1+t$ with $B_1-1-q$, and its neighbors with $M_0+t+1$ and
$M_1^1+q+t$, respectively, we reduce $M_1^1+q+t$ since its entry is
now $\geq 1+q+t$.  We continue inductively until we reach the end of the block or
$M_1^j+q+t$ which is less than $1+q+t$ since $M_1^j=0$ by
assumption.  At this point, the block looks like
$$ M_0 + t+1~|~ B_1-q, M_1^1, \dots, M_1^{j-1}-1, q+t, M_1^{j+1}+t,\dots,
M_1^k+t~|~ B_2 + t.$$
Since $B_2 +t \geq 1 +q+t$, we can reduce this block further as
$$ M^0 + t+1~|~ B_1-q, M_1^1, \dots, M_1^{j-1}-1, q+t, M_1^{j+1}+t,\dots, M_1^k+t+1~|~ B_2 -1 -q.$$
By propagating the same reductions to the rest of the configuration,
we reduce to a configuration $C^\prime$ which is made up of blocks
of the form
$$B_p-q, M_p^1, \dots, M_p^{j_p-1}-1, q+t, M_p^{j_p+1}+t,\dots,
M_p^{k_p}+t+1$$ in lieau of $$B_p, M_p^1,\dots, M_p^{j_p-1}, 0,
M^{j_p+1},\dots, M^{k_p}.$$ Since $M_p^i \leq q$, all elements of
$C^\prime$ are less than $1+q+t$ except possibly for the last
elements of each block, e.g. $M_p^k+t+1$.  If all of the $M_p^k$'s
are less than $q$, then $C^\prime$ is stable, and thus the original
configuration $C$ is not recurrent, nor critical as assumed.

Thus, without loss of generality, assume that $M_1^k = q$.  We then
can reduce block
$$ M^0 + t+1~|~ B_1-q, M_1^1, \dots, M_1^{j-1}-1, q+t, M_1^{j+1}+t,M_1^{j+2}+t\dots, M_1^{k-1}+t, q+t+1~|~ B_2 -1
-q$$ and obtain
$$ M^0 + t+1~|~ B_1-q, M_1^1, \dots, M_1^{j-1}-1, q+t, M_1^{j+1}+t,M_1^{j+2}+t\dots, M_1^{k-1}+t+1, 0~|~ B_2-1.$$
By analogous logic, we must have that $M_1^{k-1}=q$ and continuing
iteratively, we reduce to
$$ M^0 + t+1~|~ B_1-q, M_1^1, \dots, M_1^{j-1}-1, q+t+1, 0,q,\dots, q, q~|~
B_2-1$$ which is equivalent to
$$ M^0 + t+1~|~ B_1-q, M_1^1, \dots, M_1^{j-1}, 0, q,q,\dots, q, q~|~
B_2-1.$$  Finally, $M^0 = M_n^{k_n}$ so we indeed obtain
$$ q~|~ B_1, M_1^1, \dots, M_1^{j-1}, 0, q,q,\dots, q, q~|~
B_2$$ after iterating over all the blocks to the right and wrapping
around.
\end{proof}

From the Proposition, it is clear that any configuration built according to the hypothesis of Theorem \ref{critcircc} is recurrent.  Stability and thus criticality follow from Lemma 
\ref{critObb}.  Furthermore, since our initial format as given in (\ref{blockstar}) is in fact that of a general stable configuration, we in fact have proven both directions of Theorem \ref{critcircc}.

We use this characterization to describe an explicit bijection between critical configurations and spanning trees.

\begin{Thm} \label{bijthm}
There exists an explicit \emph{bijection} between critical
configurations and spanning trees for the $(q,t)$-wheel graphs, thereby inducing a group structure onto the set of spanning trees of $W_k(q,t)$.

Specifically pick one of the vertices on the rim to be $v_1$, and
label $v_2$ through $v_k$ clockwise.  Label the central hub as
$v_0$.  For $i$ between $1$ and $k$, if $1 \leq c_i \leq q$, then
fill in the arc between $v_{i-1}$ and $v_{i}$, labeling it with the
number $c_i$. (In the case of $i=1$ we use the arc between $v_k$ and
$v_1$ instead.) If $1+q \leq c_i \leq q+t$ then fill in the spoke
between $v_0$ and $v_i$ and label it with number $c_i$.  After
filling in the edges as indicated we will get a subgraph of a
spanning tree.  To complete this subgraph to a tree, fill in
additional arcs using the following rule: one may fill in an arc
from $v_{i-1}$ to $v_i$, and label it with a $q$, if and only if
$c_i \in \{1+q, \dots, q+t\}$.  In other words, if $c_i = 0$ then
this coordinate contributes no arc nor a spoke.
\end{Thm}

\begin{proof}
We start by filling in spoke $(v_0,v_i)$ for each $c_i$ satisfying $c_i\geq 1+q$.  We also label this spoke appropriately with an element in $\{1+q,\dots, q+t\}$.  For each such $c_i$, consider the block 
$c_i,M_i^1,\dots, M_i^\ell$ or $c_i,M_i^1,\dots, M_i^\ell,0,q,q,\dots,q$ where $ 1 \leq M_i^j \leq q$.  Notice that $M_i^1,\dots, M_i^\ell$ corresponds to an arc extending clockwise from the associated spoke, and each edge $(v_j,v_{j+1})$ is given a label from the set $\{1,2,\dots, q\}$.  Finally, whenever $c_i=0$, there is no arc or spoke in the tree corresponding to that coordiante.  However, we do fill out the graph into a tree on all vertices by choosing rim edges which lie clockwise from a coordinate of zero, but counter-clockwise from a coordinate greater than $q$.  Such edges only have the label of $q$, and thus we recover the definition of one possible counter-clockwise edge between a given pair of consecutive rim vertices.  Since this map is injective whose image has the correct cardinality, we have the desired bijection.
\end{proof}

\begin{Rem} After discovering the above bijection, the author learned of the Biggs-Winler \cite{BiggsWink} bijection via the burning algorithm for the case of undirected simple graphs.  When we set $(q,t)=(1,1)$, we do indeed recover the undirected simple graphs $W_k$ for which the above bijection and the Biggs-Winkler algorithm agree.\end{Rem}

\section{The Frobenius Map and Elliptic Cyclotomic Polynomials}

One of the fundamental properties of an elliptic curve over a finite field is 
the existence of the Frobenius map.  In particular, for a finite field $\f_q$, 
where $q=p^k$, $p$ prime, the Galois group $Gal(\f_{q^\ell} / \f_q)$ is cyclic generated by 
the map $\pi : x\mapsto x^q$.  (In fact $Gal(\f_{q} / \f_p)$ is also cyclic and 
generated by the analogous map, $x \mapsto x^p$, but in this paper, we will 
always be using the map which fixes ground field $\f_q$.)  This map induces an 
associated map on varieties.  Namely letting $\overline{\f_q}$ denote the 
algebraic closure of $\f_q$, by abuse of notation we also let $\pi$ denote
the map on elliptic curves.  \begin{eqnarray*} 
\pi : E(\overline{\f_q}) &\rightarrow& E(\overline{\f_q}) \\
                      (x,y) &\mapsto& (x^q,y^q)
\end{eqnarray*}

We summarize here some well-known facts about the Frobenius map on elliptic curves.  An elliptic curve can be given a group structure; for example see \cite{Silver} or \cite{Wash}.  The identity of this group is the point at infinity, which we denote as $P_\infty$.
\begin{Lem}
 \begin{eqnarray*}
 \pi(P\oplus Q) &=& \pi(P) \oplus \pi(Q), \\
\pi^k(P) = p &\mathrm{~if~and~only~if~}& P \in E(\f_{q^k}), \mathrm{~~~and}
  \end{eqnarray*}
$$E(\f_q)\subset E(\f_{q^{k_1}}) \subset E(\f_{q^{k_2}})\subset E(\f_{q^{k_3}}) \subset \cdots \subset E(\overline{\f_q})$$ whenever we have the divisibilities $k_1|k_2$, $k_2|k_3$, and so on.
\end{Lem}

\begin{proof}
See \cite{Silver}.
\end{proof}

Using this Lemma, we showed in \cite{Mus} that the equation $N_k(q,N_1)= ~$Ker$(1-\pi^k)$ can be factored, such that the left-hand-side factors into integral irreducibles simultaneously as $(1-\pi^k)$ factors into cyclotomic polynomials with respect to $\pi$.
In particular we can get an entire sequence of such factors.
\begin{Prop} [Proposition $12$ of \cite{Mus}] \label{EcycIntro}
There exists a family of bivariate irreducible integral polynomials, indexed by positive integers, which we denote as $ECyc_d(q,N_1)$ such that
$$N_k(q,N_1) = \prod_{d|k} ECyc_d(q,N_1)$$ for all $k\geq 1$.
\end{Prop}

We refer to these polynomials as \emph{elliptic cyclotomic polynomials}, and observe the following geometric interpretation.

\begin{Thm} [Theorem $7$ of \cite{Mus}] \label{geomEcyc} For all $d\geq 1$,
$$ECyc_d = \bigg|Ker\bigg(Cyc_d(\pi)\bigg): ~ E(\overline{\f_q}) \rightarrow E(\overline{\f_q}) \bigg|.$$ 
\end{Thm}

\begin{proof}
Proposition \ref{EcycIntro} and Theorem \ref{geomEcyc} both follow from the above factorization with respect to cyclotomic polynomials.  For details, see \cite{Mus}.
\end{proof}

\subsection{Analogues of Elliptic Cyclotomic Polynomials for Wheel Graphs}
Since $$N_k = \prod_{d|k} ECyc_d(q,N_1)$$ and
$\mathcal{W}_k(q,t)=-N_k\bigg|_{N_1\rightarrow -t}$, it also makes
sense to consider the decomposition
$$\mathcal{W}_k(q,t) = \prod_{d|k} WCyc_d(q,t)$$ where $WCyc_1(q,t) = t$, and 
$WCyc_d(q,t) = ECyc_d|_{N_1 \rightarrow -t}$ for $d\geq 2$. 
A few of the first several $WCyc_d(q,t)$'s are given below:
\begin{eqnarray*}
WCyc_1 &=& t \\
WCyc_2 &=& t+ 2(1+q) \\
WCyc_3 &=& t^2 + (3+3q)t + 3(1+q+q^2) \\
WCyc_4 &=& t^2 + (2+2q)t + 2(1+q^2) \\
WCyc_5 &=& t^4 + (5+5q)t^3 + (10+15q+10q^2)t^2
+ (10+15q+15q^2+10q^3)t + 5(1+q+q^2+q^3+q^4) \\
WCyc_6 &=& t^2 + (1+q)t + (1 - q + q^2) \\
WCyc_8 &=& t^4 + (4+4q)t^3 + (6+8q+6q^2)t^2 + (4+4q+4q^2+4q^3)t
+ 2(1+q^4) \\
WCyc_9 &=& t^6 + (6+6q)t^5 + (15+24q+15q^2)t^4 + (21+36q+36q^2+21q^3)t^3 \\
&+& (18+27q+27q^2+27q^3+18q^4)t^2 + (9+9q+9q^2+9q^3+9q^4+9q^5)t + 3(1+q^3+q^6) \\
WCyc_{10} &=& t^4 + (3+3q)t^3 + (4+3q+4q^2)t^2 + (2+q+q^2+2q^3)t+ (1-q+q^2-q^3+q^4) \\
WCyc_{12} &=& t^4 + (4+4q)t^3 + (5+8q+5q^2)t^2 + (2+2q+2q^2+2q^3)t +
(1-q^2+q^4)
\end{eqnarray*}
\begin{Ques} \label{QuesWCyc}
Is there an analogoue of Theorem \ref{geomEcyc} for the family of $WCyc_d(q,t)$'s?
\end{Ques}

The answer to this question is the inspiration for the next section.

\section{Maps between Critical Groups}

Fix integers $q\geq 0$ and $t\geq 1$ for this section.  Our goal is now to understand the sequence of $\bigg\{K\bigg(W_k(q,t)\bigg)\bigg\}_{k=1}^\infty$ in a way that corresponds to the chain $$E(\f_q)\subset E(\f_{q^{k_1}}) \subset E(\f_{q^{k_2}})\subset E(\f_{q^{k_3}}) \subset \cdots \subset E(\overline{\f_q})$$ for $k_1|k_2|k_3$, etc.

\begin{Prop} \label{Proppp1}
The identity map induces an injective group homomorphism between
$K(W_{k_1}(q,t))$ and $K(W_{k_2}(q,t))$ whenever
$k_1|k_2$.  More precisely, we let \\ $K(W_{k_1}(q,t))$
embed into $K(W_{k_2}(q,t))$ by letting $w \in
K(W_{k_1}(q,t))$ map to the word $www\dots w \in
K(W_{k_2}(q,t))$ using $\frac{k_2}{k_1}$ copies of $w$.
\end{Prop}

\noindent Define $\rho$ to be the rotation map on
$K(W_k(q,t))$.  If we consider elements of the critical
group to be configuration vectors, then we mean circular rotation of
the elements to the left.  On the other hand, $\rho$ acts by
rotating the rim vertices of $W_k$ counter-clockwise if we view elements of
$K(W_k(q,t))$ as spanning trees.

\begin{Prop} \label{Proppp2}
The kernel of $(1-\rho^{k_1})$ acting on $K(W_{k_2}(q,t))$
is subgroup $K(W_{k_1}(q,t))$ whenever $k_1|k_2$.
\end{Prop}

\begin{proof}
We prove both of these propositions simultaneously, by noting that
chip firing is a local process.  Namely, if $k_1$ divides $k_2$ and
we add two configurations of $W_{k_1}(q,t)$ together pointwise to
get configuration $C$, then lift $C$ to a length $k_2$ configuration
$C^\prime$ of $W_{k_2}(q,t)$ by periodically extending length $k_1$
vector $C$.  Then the claim is that if $C$ reduces to unique
critical configuration $\overline{C}$, then $C^\prime$ also reduces
to $\overline{C}$'s periodic extension.  To see this, observe that
every time vertex $v \in W_{k_1}(q,t)$ fires in the reduction
algorithm, then we could simultaneously fire the set of vertices of
$W_{k_2}(q,t)$ in the image of $v$ after lifting.  In other words,
if $v_i \in W_{k_1}(q,t)$ fires, we fire $\{v_i^\prime,
v_{i+{k_2/k_1}}^\prime, v_{i+2{k_2/k_1}}^\prime,\dots\} \in
W_{k_2}(q,t)$ thus obtaining the lift of the configuration reached
after $v$ fires.
\end{proof}

\begin{figure} [hpbt] Example: $[2,4,2] \oplus [0,4,1] \equiv [1,0,4]$ in
$W_3(q=3,t=2)$ versus
\\
\includegraphics[width = 1.5in, height = 1.25in]{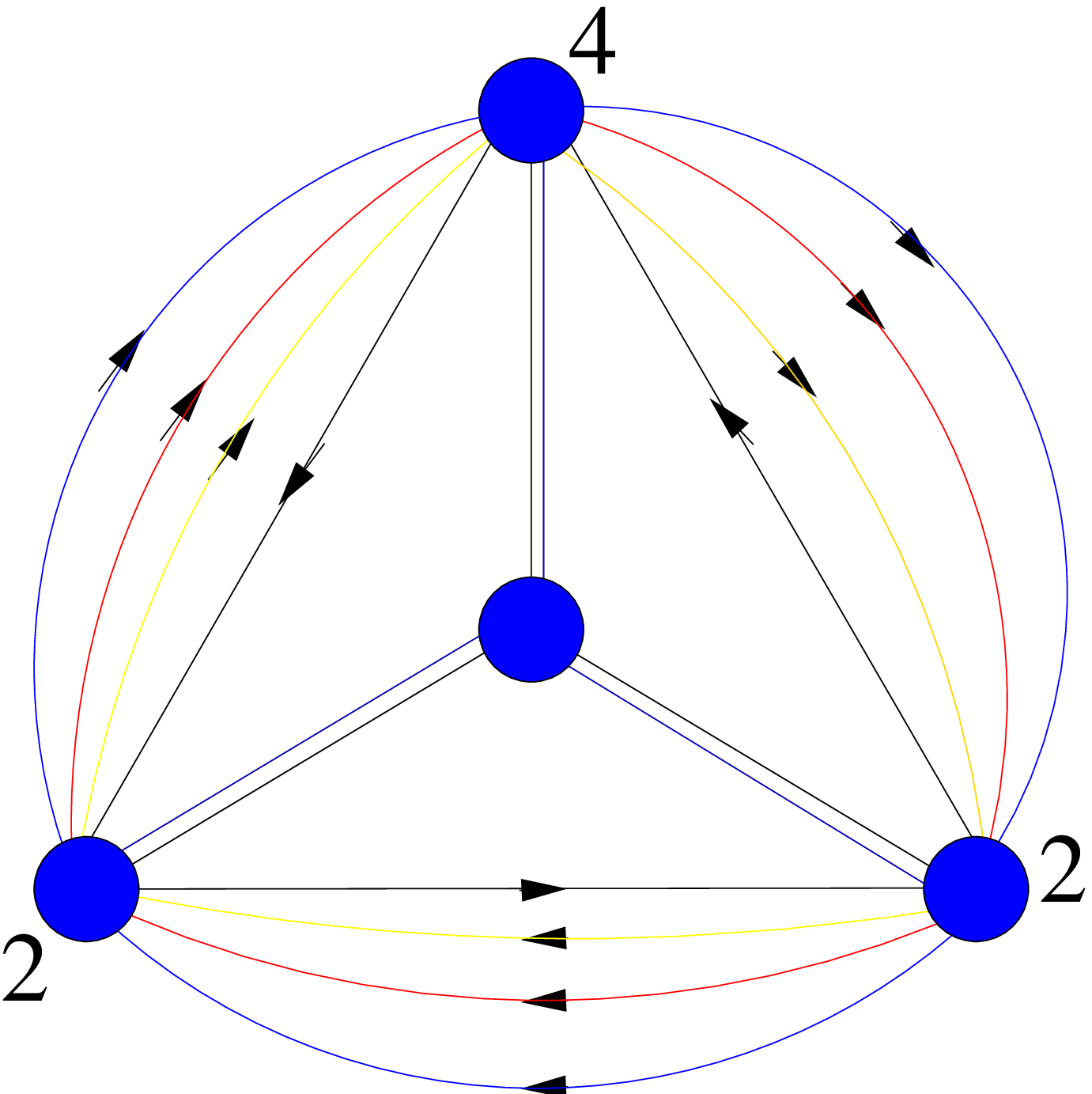}
$\oplus$ \includegraphics[width = 1.5in, height =
1.25in]{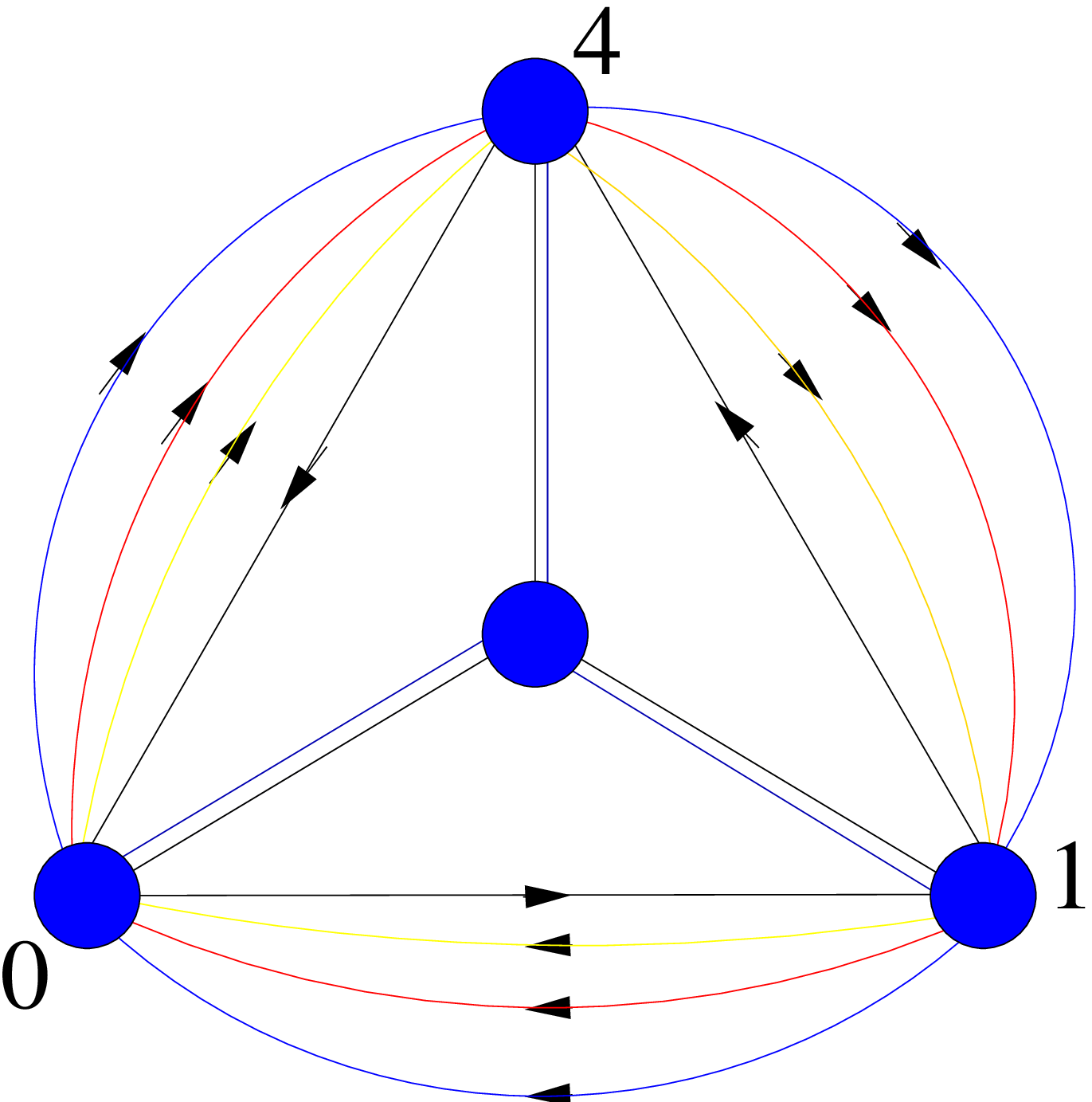} $=$ \includegraphics[width = 1.5in, height
=1.25in]{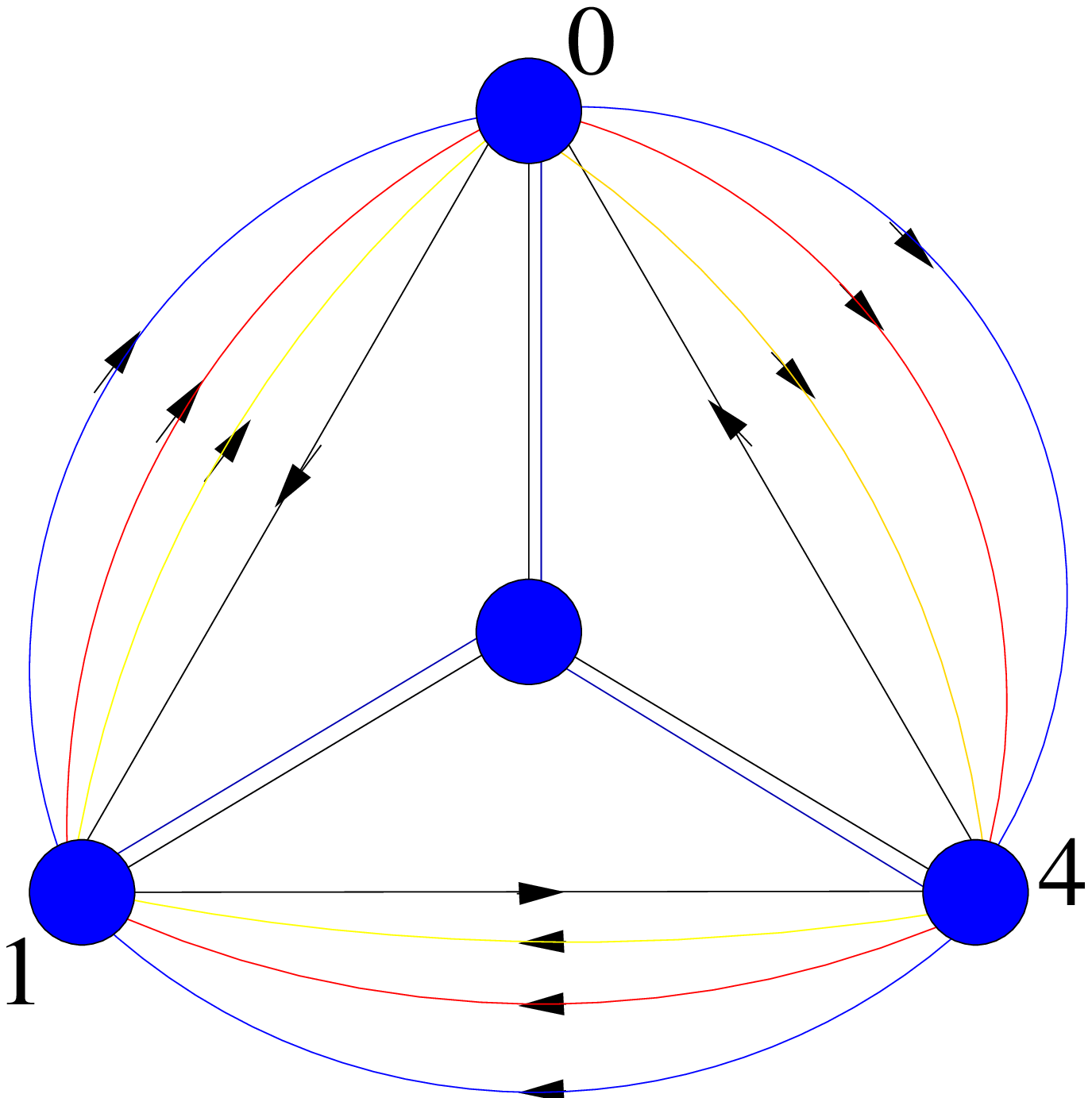}
\\
$[2,4,2,2,4,2] \oplus [0,4,1,0,4,1] \equiv [1,0,4,1,0,4]$ in
$W_6(q=3,t=2)$
\\
\includegraphics[width = 1.5in, height = 1.25in]{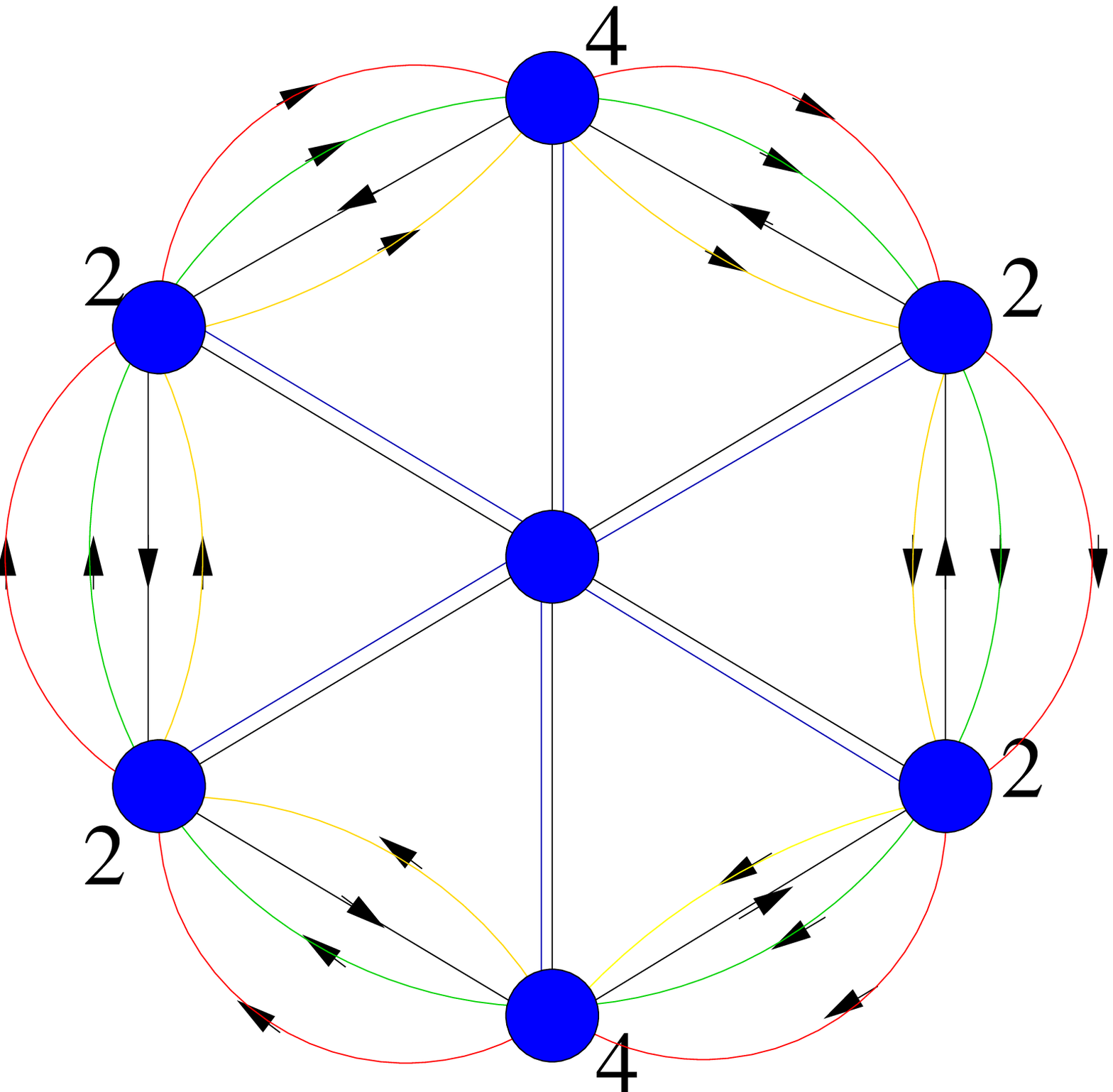}
$\oplus$ \includegraphics[width = 1.5in, height =
1.25in]{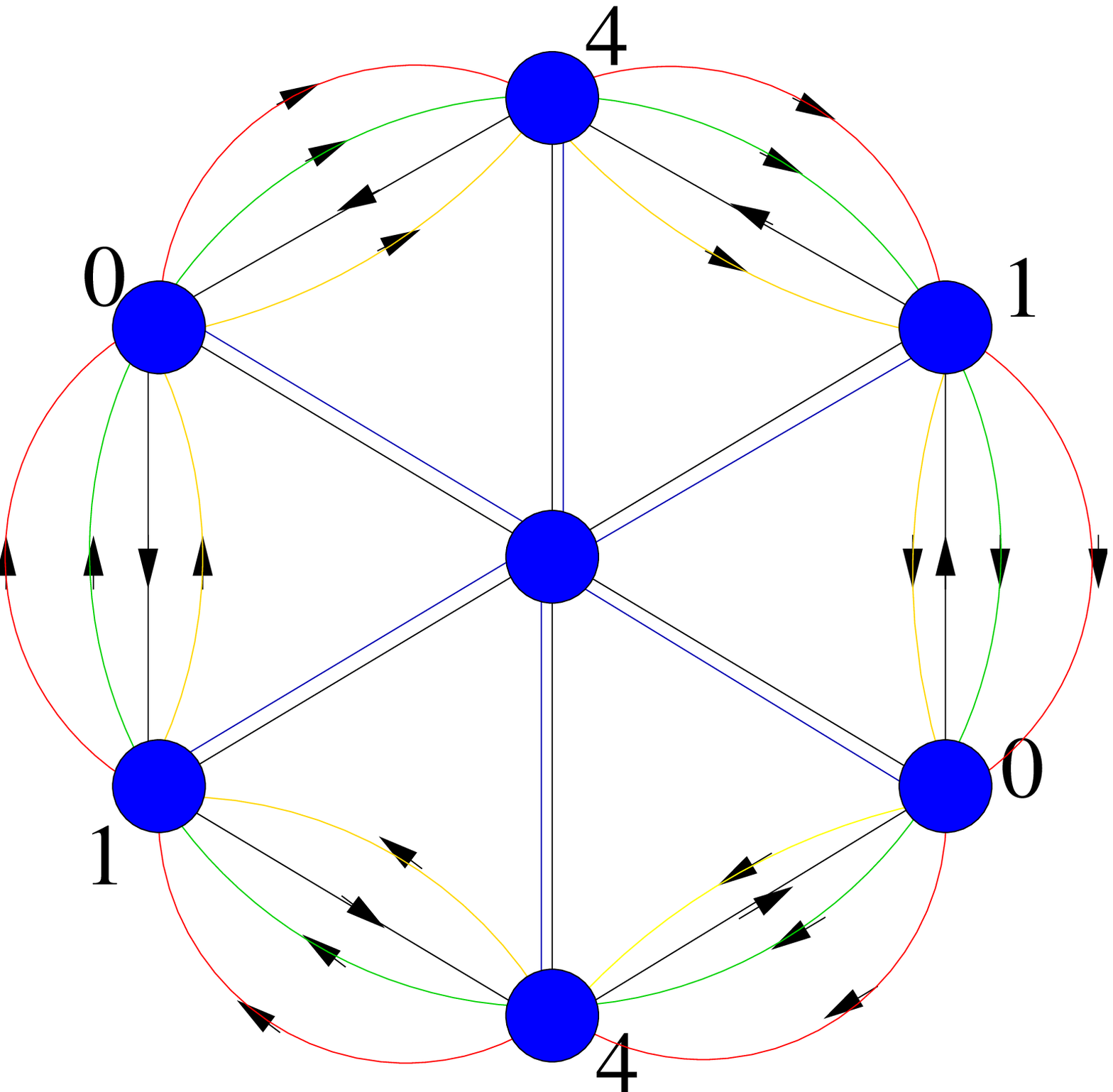} $=$ \includegraphics[width = 1.5in, height
= 1.25in]{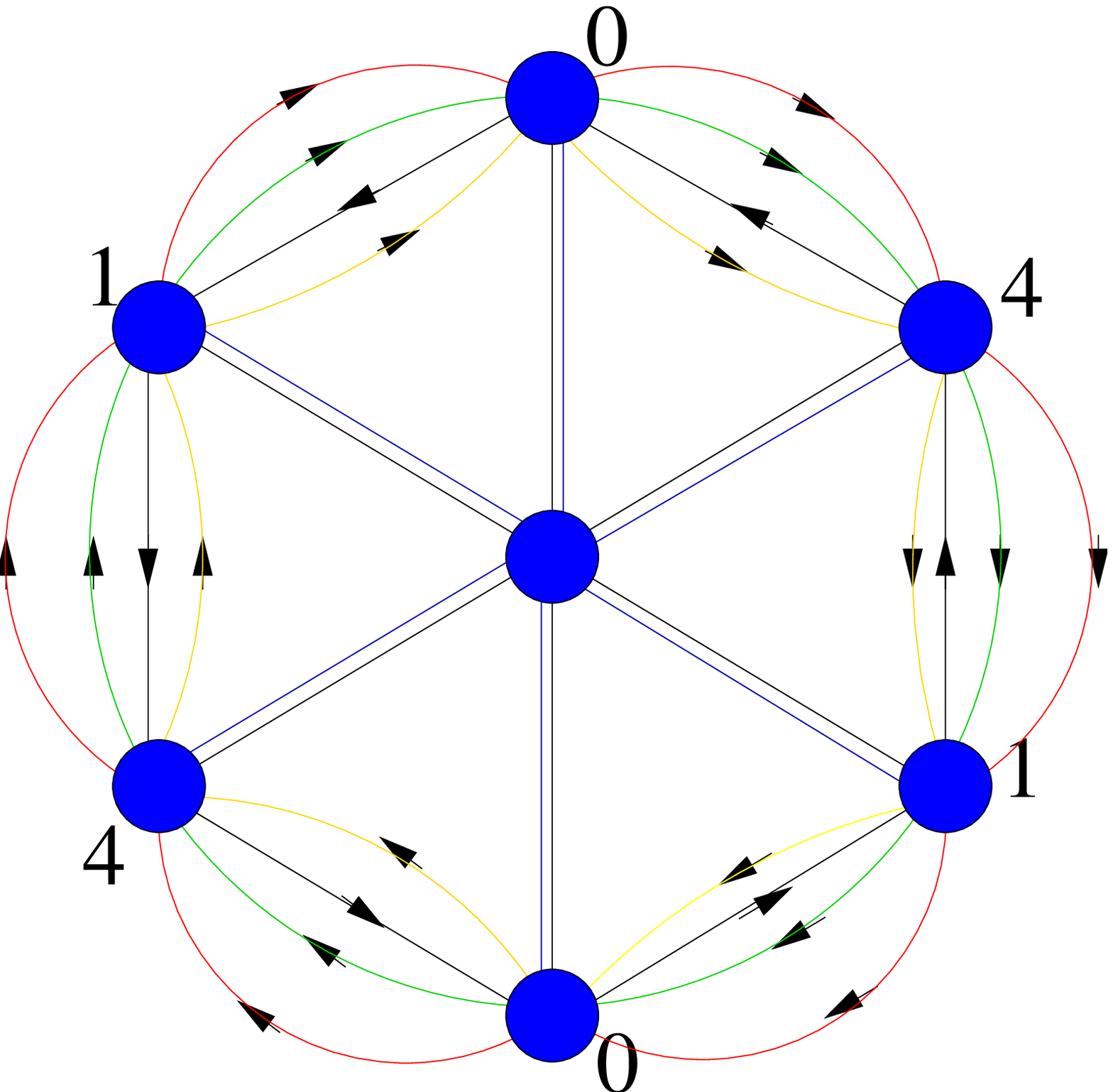} \caption{Illustrating Propositions
\ref{Proppp1} and \ref{Proppp2}.}\end{figure}
\noindent We therefore can define a direct limit
$$K(\overline{W}(q,t)) \cong \bigcup_{k=1}^\infty
K(W_k(q,t))$$ where $\rho$ provides the transition maps.

Another view of $K(\overline{W}(q,t))$ is as the set of
bi-infinite words which are (1) periodic, and (2) have fundamental
subword equal to a configuration vector in $K(W_k(q,t))$
for some $k \geq 1$.  In this interpretation, map $\rho$ acts on
$K(\overline{W}(q,t))$ also.  In this case, $\rho$ is the
shift map, and in particular we obtain $$K(W_k(q,t)) \cong
Ker (1-\rho^k): K(\overline{W}(q,t)) \rightarrow
K(\overline{W}(q,t)).$$ We now can describe our variant of
Theorem \ref{geomEcyc}.

\begin{Thm} \label{geomWcyc}
$$WCyc_d = \bigg|Ker\bigg(Cyc_d(\rho)\bigg): ~ K(\overline{W}(q,t)) \rightarrow K(\overline{W}(q,t)) \bigg|$$ where
$\rho$ denotes the shift map, and $K(\overline{W}(q,t))$ is the direct limit of the sequence $\{K(W_k(q,t))\}_{k=1}^{\infty}$.
\end{Thm}

\begin{proof}
The proof is analogous to the elliptic curve case.  Since the maps
$Cyc_{d_1}(\rho)$ and $Cyc_{d_2}(\rho)$ are group homomorphisms, we
get $$\bigg|\mathrm{Ker}~\bigg(Cyc_{d_1}(\rho)~ Cyc_{d_2}(\rho)\bigg)\bigg| =
|\mathrm{Ker}~ Cyc_{d_1}(\rho)|\cdot
|\mathrm{Ker}~Cyc_{d_2}(\rho)|$$ and the rest of the proof follows
as in \cite{Mus}.
\end{proof}

Consequently we identify shift map $\rho$ as being the analogue of the Frobenius map $\pi$ on elliptic curves.  In addition to $\rho$'s appearance in
\begin{eqnarray*} K(W_k(q,t)) &\cong& Ker (1-\rho^k): K(\overline{W}(q,t)) \rightarrow K(\overline{W}(q,t))
\mathrm{~~~just~as} \\
E(\f_{q^k}) &=& Ker (1-\pi^k): E(\overline{\f_q}) \rightarrow E(\overline{\f_q}).
\end{eqnarray*}
and Theorem \ref{geomWcyc}, another comparison with $\pi$ is highlighted below.
\begin{Thm} \label{quadwheel} As a map from $K(\overline{W}(q,t))$ to itself, we get $$\rho^2 - (1+q+t)\rho +q =0.$$
\end{Thm} 

\noindent Note that this quadratic is a simple analogue of the characteristic equation
$$\pi^2 - (1+q-N_1)\pi + q = 0$$ of the Frobenius map $\pi$.  In the case of elliptic curves, this equation is proven using an analysis of the endomorphism ring of an elliptic ring or the Tate Module.  However in the critical group case, linear algebra suffices.
\begin{proof} [Proof of Theorem \ref{quadwheel}]
Since elements of $K(\overline{W}(q,t))$ are periodic extensions of some vector in $K(W_k(q,t))$ for some $k$, it suffices to prove the identity on $K(W_k(q,t))$ for all $k\geq 1$. In particular, if  $\rho(C)=[c_2, c_3,\dots, c_{k}, c_{1}]$, 
then we notice that 
$\rho^2(C) - (1+q+t)\rho(C) + q\cdot C$ equals 
\begin{small}
$$c_1\left[\begin{matrix} q \\ 0 \\ 0\\ \vdots \\0 \\ 1 \\ -(1+q+t) \end{matrix}\right]^T
+  c_2\left[\begin{matrix} -(1+q+t) \\ q \\ 0 \\ \vdots \\ 0 \\ 0 \\ 1 \end{matrix}\right]^T
+ \dots 
+ c_k\left[\begin{matrix} 0 \\0\\ 0\\ \vdots \\ 1 \\ -(1+q+t) \\ q \end{matrix}\right]^T,$$
\end{small}
which equals $[0,0,0,0,\dots, 0,0]$ modulo the rows of the reduced Laplacian matrix.
\end{proof}

An even more surprising connection is the subject of the next subsection.

\subsection{Group Presentations}

It is well known that an elliptic curve over a finite field has a group structure which is the product of at most two cyclic groups.  This result can be seen as a manifestation of the lattice structure of an elliptic curve over the complex numbers.  One way to
prove this is by showing that for $gcd(N,p) = 1$, the $[N]$-torsion
subgroup of $E(\overline{\f}_p)$ (also denoted as $E[N]$) is
isomorphic to $\z / N\z \times \z / N\z$ and that $E[p^r]$ is either
$0$ or $\z / p^r \z$.

Since we know that the critical group of graphs are also abelian
groups, this motivates the question: what is the group decomposition
of the $K(G)$'s?  The case of a simple wheel graph $W_k$
was explicitly found in \cite{Biggs} to be $$\z / L_k \z \times \z /
L_k \z \mathrm{~~~or~~~} \z / F_{k-1} \z \times \z / 5 F_{k-1}\z$$
depending on whether $k$ is odd or even, respectively. Here $L_k$ is
the $k$th Lucas number and $F_k$ is the $k$th Fibonacci number.

Determining such structures of critical groups has been the subject
of several papers recently, e.g. \cite{Vic1,Molly}, and a common
tool is the Smith normal form of the Laplacian.  We use the same method here to prove the following comparison of elliptic curves and critical groups.

\begin{Thm} \label{twogp}
$K(W_k(q,t))$ is isomorphic to at most two cyclic
groups, a property that this sequence of critical groups shares with
the family of elliptic curve groups over finite fields.
\end{Thm}

\begin{proof}
The Smith normal form of a matrix is unchanged by

\begin{enumerate}

\item
Multiplication of a row or a column by $-1$.

\item
Addition of an integer multiple of a row or column to another.

\item
Swapping of two rows or two columns.

\end{enumerate}

We let $M_k$ be the $k$-by-$k$ circulant matrix $circ(1+q-N_1,-q,0,\dots, 0,-1)$ as in Section $2$, and let $\overline{M}_k$ denote the $k$-by-$k$ matrix 
$circ(1+q+t,-q,0,\dots, 0,-1)$, the reduced Laplacian of the $(q,t)$-wheel graph.  To begin we note after permuting rows cyclically and multiplying
through all rows by $(-1)$ that we get

$$\overline{M}_k^T \equiv \left[ \begin{matrix}
1 & 0 & \dots &0 & q &-1-q-t \\
-1-q-t & 1 &  0 & \dots & 0& q \\
q & -1-q-t & 1 &  0 & \dots & 0 \\
\dots & \dots & \dots & \dots & \dots & \dots\\
\dots & 0 & q & -1-q-t & 1 & 0 \\
0 & \dots & 0 & q & -1-q-t & 1
\end{matrix}\right].$$

Except for an upper-right corner of three nonzero entries, this matrix is lower-triangular with ones on the diagonal.  Adding a multiple of
the first row to the second and third rows, respectively, we obtain
a new matrix with vector $$[1,0,0,\dots, 0]^T$$ as the first
column. Since we can add multiples of columns to one another as
well, we also obtain a matrix with vector $[1,0,0,\dots, 0]$ as the first row.

This new matrix will again be lower triangular with ones along the
diagonal, except for nonzero entries in four spots in the last two
columns of rows two and three.  By the symmetry and sparseness of
this matrix, we can continue this process, which will always shift
the nonzero block of four in the last two columns down one row. This
process will terminate with a block diagonal matrix consisting of
($k-2$) $1$-by-$1$ blocks of element $1$ followed by a single
$2$-by-$2$ block. \end{proof}

Since this proof is construcutive, as an application we arrive at a presentation of the $K(W_k(q,t))$'s as the cokernels of a $2$-by-$2$ matrix whose entries have combinatorial interpretations.

\begin{Cor} \label{2Cyclic}
For $k\geq 3$, the Smith normal form of $\overline{M}_k$ is equivalent to a direct sum of the identity matrix and 
$$\left[ \begin{matrix}
q \hat{F}_{2k-4} +1 & q\hat{F}_{2k-2} \\
\hat{F}_{2k-2} & \hat{F}_{2k}-1
\end{matrix}\right]$$ 
where $\hat{F}_{2k}(q,t)$ is defined as 
$$\hat{F}_{2k}(q,t)= \sum_{S \subseteq \{1,2,\dots,2k\}~~:~~ S\mathrm{~contains~no~two~consecutive~elements}}
q^{\#\mathrm{~even~elements~in~}S}~t^{k-\#S}.$$  Notice that these are a bivariate analogue of the Fibonacci numbers.
\end{Cor}

\begin{Rem} 
An analogous family of polynomials, namely the $E_k$'s, defined by $$E_k(q,N_1)=(-1)^k \sum_{S \subseteq \{1,2,\dots,2k-2\}~~:~~ S\mathrm{~contains~no~two~consecutive~elements}}
q^{\#\mathrm{~even~elements~in~}S}~(-N_1)^{k-\#S}$$ appeared in \cite{Mus} where they had a plethystic interpretation as $e_k[1+q-\alpha_1-\alpha_2]$ such that $\alpha_1$ and $\alpha_2$ are the two roots of $qT^2 - (1+q-N_1)T + 1$.
\end{Rem}

Before giving the proof of Corollary \ref{2Cyclic}, we show the following more general result. Define $\widetilde{M_k}$ as the following $k$-by-$k$ matrix:

$$\widetilde{M_k} = \left[ \begin{matrix}
1 & 0 & 0 & \dots & 0 & 0 & A & B \\
-\Delta & 1 & 0 & \dots & 0 & 0 & C & D \\
q & -\Delta & 1 & \dots & 0 & 0 & 0 & 0 \\
\dots & \dots & \dots & \dots & \dots & \dots & \dots\\
0 & 0 & 0 & \dots & 1 & 0 & 0 & 0 \\
0 & 0 & 0 & \dots & -\Delta & 1 & 0 & 0  \\
0 & 0 & 0 & \dots & q & -\Delta & W & X \\
0 & 0 & 0 & \dots & 0 & q & Y & Z
\end{matrix}\right].$$

\begin{Prop} \label{GenSmith}
The Smith normal form of $\widetilde{M_k}$ is equivalent to
$$\left[ \begin{matrix}
1 & 0 & \dots & 0 & 0 & 0 \\
0 & 1 & \dots & 0 & 0 & 0 \\
\dots & \dots & \dots & \dots & \dots & \dots\\
0 & 0 & \dots & 1 & 0 & 0 \\
0 & 0 & \dots & 0 & a & b \\
0 & 0 & \dots & 0 & c & d
\end{matrix}\right]$$
where  
$\left[ \begin{matrix}
a & b \\
c & d \end{matrix}\right] =  \left[ \begin{matrix}
\Delta & 1 \\
-q & 0
\end{matrix}\right]^{k-2} \left[ \begin{matrix}
A & B \\
C & D
\end{matrix}\right] + \left[ \begin{matrix}
W & X \\
Y & Z
\end{matrix}\right]
$.
\end{Prop}

\begin{proof}
We represent the last two columns of $\widetilde{M_k}$ as
$\left[ \begin{matrix} a_1^{\prime\prime} & b_1^{\prime\prime} \\
a_2^\prime & b_2^\prime \\
a_3 & b_3 \\
a_4 & b_4 \\
a_5 & b_5 \\
\vdots & \vdots \\
a_k & b_k
\end{matrix}\right],$ and letting 
$\left[ \begin{matrix} 0 & 0 \\
a_2^{\prime\prime} & b_2^{\prime\prime} \\
a_3^\prime & b_3^\prime \\
a_4 & b_4 \\
a_5 & b_5 \\
\vdots & \vdots \\
a_k & b_k
\end{matrix}\right]$ signify the last two columns after completing the steps outlined above, i.e. subtracting a multiple of the first row from the second and third row, and then using the first column to cancel out the entries $a_1^{\prime\prime}$ and $b_1^{\prime\prime}$.

Continuing inductively, we get the relations
\begin{eqnarray*}
 a_m^{\prime\prime} &=& \Delta a_{m-1}^{\prime\prime} + a_m^\prime \\
 b_m^{\prime\prime} &=& \Delta b_{m-1}^{\prime\prime} + b_m^\prime \\
 a_{m+1}^{\prime} &=& q a_{m-1}^{\prime\prime} + a_{m+1} \\
 b_{m+1}^{\prime} &=& q b_{m-1}^{\prime\prime} + b_{m+1},
\end{eqnarray*}
which we encode as the matrix equation 
$$\left[ \begin{matrix}
a_m^{\prime\prime} & b_m^{\prime\prime} \\
a_{m+1}^\prime & b_{m+1}^\prime \end{matrix}\right] =  \left[ \begin{matrix}
\Delta & 1 \\
-q & 0
\end{matrix}\right] \left[ \begin{matrix}
a_{m-1}^{\prime\prime} & b_{m-1}^{\prime\prime} \\
a_m^\prime & b_m^\prime
\end{matrix}\right] + \left[ \begin{matrix}
0 & 0 \\
a_{m+1} & b_{m+1}
\end{matrix}\right].
$$
Letting $a_3,b_3,\dots, a_{k-2},b_{k-2}=0$ and using 
$\left[ \begin{matrix}
\Delta & 1 \\
-q & 0
\end{matrix}\right] \left[ \begin{matrix}
0 & 0 \\
W & X
\end{matrix}\right] + \left[ \begin{matrix}
0 & 0 \\
Y & Z
\end{matrix}\right]
= \left[ \begin{matrix}
W & X \\
Y & Z
\end{matrix}\right]$, we obtain the desired result.
\end{proof}

We now wish to consider the special case $\Delta = 1+q+t$,
$\left[ \begin{matrix}
A & B \\
C & D \end{matrix}\right] = \left[ \begin{matrix}
q & -\Delta \\
0 & q \end{matrix}\right]$, and
$\left[ \begin{matrix}
W & X \\
Y & Z \end{matrix}\right] = \left[ \begin{matrix}
1 & 0 \\
-\Delta & 1 \end{matrix}\right]$.
To simplify our expression further, we utlize the following formula for a specific sequence of matrix powers.

\begin{Lem} \label{LemPow}
For all $m\geq 2$, $$\left[\begin{matrix}
1+q+t & 1 \\
-q & 0
\end{matrix}\right]^m = \left[
\begin{matrix}
\hat{F}_{2m} & \hat{F}_{2m-2} \\
-q\hat{F}_{2m-2} & -q\hat{F}_{2m-4}
\end{matrix}\right].$$
\end{Lem}

\begin{proof}
We verify the result for $m=2$ using the fact that 
\begin{eqnarray*}
\hat{F}_0 &=& 1 \\
\hat{F}_2 &=& t + (1+q) \\
\hat{F}_4 &=& t^2 + (2+2q)t + (1+q+q^2) = (1+q+t)^2 - q.
\end{eqnarray*}
The product $\left[\begin{matrix}1+q+t & 1 \\
-q & 0
\end{matrix}\right]\left[
\begin{matrix}
\hat{F}_{2m} & \hat{F}_{2m-2} \\
-q\hat{F}_{2m-2} & -q\hat{F}_{2m-4}
\end{matrix}\right]$ equals
$$\left[
\begin{matrix}
(1+q+t)\hat{F}_{2m} - q\hat{F}_{2m-2} & 
(1+q+t)\hat{F}_{2m-2} - q\hat{F}_{2m-4}  \\ \\
-q\hat{F}_{2m} & -q\hat{F}_{2m-2}
\end{matrix}\right].$$  Thus it suffices to demonstrate 
$$\hat{F}_{2m+4}=(1+q+t)\hat{F}_{2m+2} - q \hat{F}_{2m}$$ by recursion.
This recurrence was proven in \cite{Mus}; the proof is a generalization of the well-known recurrence $F_{2m+4} = 3F_{2m+2} - F_{2m}$ for the Fibonacci numbers.

Namely, the polynomial $\hat{F}_{2m+4}$ is a $(q,t)$-enumeration of the
number of chains of $2m+4$ beads, with each bead either black or
white, and no two consecutive beads both black.  Similarly
$(1+q+t)\hat{F}_{2m+2}$ enumerates the concatenation of such a chain of
length $2m+2$ with a chain of length $2$.  One can recover a legal
chain of length $2m+4$ this way except in the case where the
$(2m+2)$nd and $(2m+3)$rd beads are both black.  Since this forces the $(2m+1)$st and $(2m+4)$th beads to be white, such cases are
enumerated by $q\hat{F}_{2m}$ and this completes the proof.
With this recurrence, Lemma \ref{LemPow} is proved.
\end{proof}

\begin{proof} [Proof of Corollary \ref{2Cyclic}.]

Here we give the explicit derivation of matrix
$\left[ \begin{matrix}
a & b \\
c & d \end{matrix}\right]$ in terms of the $\hat{F}_k(q,t)$'s.
By Proposition \ref{GenSmith} and Lemma \ref{LemPow}, we let $m=k-2$ and we obtain 
\begin{eqnarray*}
\left[ \begin{matrix}
a & b \\
c & d \end{matrix}\right] &=&  \left[ \begin{matrix}
\hat{F}_{2k-4} & \hat{F}_{2k-6} \\
-q\hat{F}_{2k-6} & -q\hat{F}_{2k-8}
\end{matrix}\right] \left[ \begin{matrix}
q & -1-q-t \\
0 & q
\end{matrix}\right] + \left[ \begin{matrix}
1 & 0 \\
-1-q-t & 1
\end{matrix}\right] \\
&=& 
\left[ \begin{matrix}
q\hat{F}_{2k-4}+1 &   -(1+q+t)\hat{F}_{2k-4} + q\hat{F}_{2k-6} \\
-q^2\hat{F}_{2k-6}-1-q-t &    (1+q+t)q\hat{F}_{2k-6} - q^2\hat{F}_{2k-8}+1
\end{matrix}\right]
\end{eqnarray*}
when reducing $\overline{M}_k^T$ to a $2$-by-$2$ matrix with an equivalent Smith normal form.

We apply the recursion $\hat{F}_{2m+4}=(1+q+t)\hat{F}_{2m+2} - q \hat{F}_{2m}$
followed by adding a multiple of $(1+q+t)$ times the first row to the second row, and then use the recursion again to get
$\left[ \begin{matrix}
q \hat{F}_{2k-4} +1 & -\hat{F}_{2k-2} \\
q\hat{F}_{2k-2} & -\hat{F}_{2k}+1
\end{matrix}\right]$.
Finally we mutliply the second row column through by $(-1)$ and take the transpose, thereby obtaining the deisred result.
\end{proof} 

If one plugs in specific integers for $q$ and $t$ ($q\geq 0, ~t\geq 1$), then one can reduce the Smith normal form further.  In general, the Smith normal form of a $2$-by-$2$ matrix  $\left[ \begin{matrix}
a & b \\
c & d \end{matrix}\right]$
can be written as $diag(d_1,d_2)$ where $d_1 = \gcd(a,b,c,d)$.
The group $K(W_k(q,t))$ is cyclic if and only if $d_1=1$ in this case.
\begin{Ques}
How can one predict what choices of $(k,q,t)$ lead to a cyclic critical group, and can we more precisely describe the group structure otherwise?
\end{Ques}
As Biggs discusses in \cite{BiggsCrypto}, being able to find families of graphs with cyclic critical groups might have cryptographic applications just as it is important to find prime powers $q$ and elliptic curves $E$ such that the groups $E(\f_q)$ are cyclic.

Answering this question for $W_k(q,t)$'s is difficult since the Smith normal form of even a $2$-by-$2$ matrix can vary wildly as the four entries change, altering the greatest common divisor along with them.  However, we give a partial answer to this question below, after taking a segway into a related family of graphs.

\begin{Rem}
In \cite{BiggsCrypto}, Biggs shows that a family of deformed wheel graphs (with an odd number of vertices) have cyclic critical groups.  We are able to obtain a generalization of this result here by using Proposition \ref{GenSmith}.  The author thanks Norman Biggs \cite{BiggsPersonal} for bringing this family of graphs to the author's attention.
\end{Rem}

Biggs defined $\widetilde{W_k}$ by taking the simple wheel graph $W_k$ with $k$ rim vertices and adding an extra vertex on one of the rim vertices.  Equivalently, $\widetilde{W_k}$ can be constructed from $W_{k+1}$ by removing one spoke.  We construct a $(q,t)$-deformation of this family by defining $\widetilde{W_k(q,t)}$ as the graph $W_{k+1}(q,t)$ where all edges, i.e. spokes, connecting vertex $v_0$ and $v_1$ are removed.

With such a deformation, it is no longer true that this entire family of graphs have cyclic critical groups, but the next theorem gives a precise criterion for cyclicity and further gives an explicit formula for the smaller of the two invariant factors otherwise.

\begin{Thm} \label{BiggsCyclicQT}
For all $k\geq 1$, let $Q_k = 1+q+q^2+\dots + q^k$.  If $\gcd(t,Q_k)=1$  then the critical group of $\widetilde{W_k(q,t)}$ is cyclic.  Otherwise, if we let $d_1 = \gcd(t,Q_k)$, then $\widetilde{W_k(q,t)} \cong \z / d_1 \z \times \z / {d_1 d_2 \z}$.
\end{Thm}

\begin{proof}
Notice, that the reduced Laplacian matrix for $\widetilde{W_k(q,t)}$ agrees with matrix $\overline{M}_{k+1}$ except in the first entry corresponding to the outdegree of $v_1$.  After taking the transpose, cyclically permuting the rows, and multiplication by $(-1)$, we obtain a matrix adhering to the hypothesis of Proposition \ref{GenSmith} with
$\left[ \begin{matrix}
A & B \\
C & D \end{matrix}\right] = \left[ \begin{matrix}
q & -1-q \\
0 & q \end{matrix}\right]$ and 
$\left[ \begin{matrix}
W & X \\
Y & Z
\end{matrix}\right] =
\left[ \begin{matrix}
1 & 0 \\
-1-q+N_1 & 1
\end{matrix}\right]$.
Thus, the $2$-by-$2$ matrix for this case equals \begin{eqnarray*}
&&\left[ \begin{matrix}
\hat{F}_{2k-2} & \hat{F}_{2k-4} \\
-q\hat{F}_{2k-4} & -q\hat{F}_{2k-6}
\end{matrix}\right] \left[ \begin{matrix}
q & -1-q \\
0 & q
\end{matrix}\right] + \left[ \begin{matrix}
1 & 0 \\
-1-q-t & 1
\end{matrix}\right] \\
&=& \left[ \begin{matrix}
q\hat{F}_{2k-2}+1 & -(1+q)\hat{F}_{2k-2}+q \hat{F}_{2k-4} \\
-q\hat{F}_{2k-4}-1-q-t & (1+q)q\hat{F}_{2k-4}-q^2 \hat{F}_{2k-6}+1
\end{matrix}\right].
\end{eqnarray*}

As in Corollary \ref{2Cyclic}, we add $(1+q+t)$ times the first row to the second row, and then multiply the second column by $(-1)$ to arrive at
$$\left[\begin{matrix} a^\prime & b^\prime \\ c^\prime & d^\prime \end{matrix}\right] = \left[ \begin{matrix}
q\hat{F}_{2k-2}+1 & (1+q)\hat{F}_{2k-2}-q \hat{F}_{2k-4} \\
q\hat{F}_{2k} & (1+q)\hat{F}_{2k}-q^2 \hat{F}_{2k-2} +1
\end{matrix}\right] = \left[ \begin{matrix}
q\hat{F}_{2k-2}+1 &  \hat{F}_{2k} - t\hat{F}_{2k-2} \\
q\hat{F}_{2k} & \hat{F}_{2k+2}-t\hat{F}_{2k} -1
\end{matrix}\right].$$

We can reduce this by plugging in specific values for $q$ and $t$ and checking whether or not $Q_k =1+q+\dots + q^k$ and $t$ share a common factor.  We know there exist a unique $d_1$ and $d_2$ such that 
$\left[\begin{matrix} a^\prime & b^\prime \\ c^\prime & d^\prime \end{matrix}\right]
=\left[\begin{matrix} d_1 & 0 \\ 0 & d_1d_2 \end{matrix}\right]$.
We begin by showing that $d_1$ must divide $t$.
Suppose otherwise; then looking at the off-diagonal entries $b^\prime$ and $c^\prime$, we see $d_1$ divides $q\hat{F}_{2k}$ and $\hat{F}_{2k}-t\hat{F}_{2k-2}$ but not $t$, and so $d_1$ must divide either $q$ or $\hat{F}_{2k-2}$.  However, $d_1$ must also divide the top left entry, which is $q\hat{F}_{2k-2}+1$.  Thus we get a contradiction, and conclude that $d_1 | t$.

This greatly limits the possibilities for $d_1$.  Furthermore, if we work modulo $t$, the equivalence class of $a^\prime,b^\prime, c^\prime$, and $d^\prime$ (modulo $d_1$) does not change.

Letting $t=0$ in $\hat{F}_{2k}$ is equivalent to counting subsets of $\{1,2,\dots, 2k\}$ of size $k$ with no two elements consecutive.  We can choose the subsets of all odds numbers, which will have weight $1$.  If we then pick element $2k$ instead of $2k-1$, we get a subset of weight $q$, and inductively, we get a weighted sum of $1+q+q^2+\dots+q^k$ where the last weight corresponds to the subset of all even numbers.
Thus the desired $2$-by-$2$ matrix reduces to
$\left[ \begin{matrix}
q(1+q+q^2+\dots+q^{k-1})+1 & 1+q+q^2+\dots+q^{k} \\
q(1+q+q^2+\dots+q^{k}) & q+q^2+\dots+q^{k+1}
\end{matrix}\right]$, hence modulo $t$, the gcd of the entries is the quantity $Q_k$.

Thus we conclude that $d_1 | Q_k$.  Combining this fact with $d_1|t$, we conclude $d_1|\gcd(Q_k,t)$.  However, since we know that $d_1 \equiv Q_k$ (mod $t$), we have integers $m_1,m_2$ such that $t= m_1d_1$, $Q_k = (m_1m_2+1)d_1$, and so $\gcd(Q_k,t)=d_1\cdot m_3$ where $m_3 = gcd(m_1,m_1m_2+1)=1$.  Thus we obtain the equality $d_1 = \gcd(Q_k,t)$ as desired.
\end{proof}

\begin{Rem}
Notice, that if $q=1$ and $t=1$ we have $d_1=1$, hence cyclicity in this case.  This result was proven by Biggs for case of odd $k$, and the above proof of Theorem \ref{BiggsCyclicQT} specializes to give an alternate proof of this result.
\end{Rem}

\begin{Rem}
The above proof can also be adapted to demonstrate values of $(k,q,t)$ for which the original critical groups, $K(W_k(q,t))$, are not cyclic.  Namely, by pushing through the same argument, we get $K(W_k(q,t))$ is \emph{not} cyclic whenever $K(\widetilde{W_k(q,t)})$ is \emph{not} cyclic.  Unfortuantely, we do not get an if and only if criterion nor a precise formula for the smaller invariant factor in this case.  This is due to the fact that there are cases where $d_1$ does not divide $t$ for the undeformed wheel graphs.  In particular, this allows the simple wheel graphs to have non-cyclic critical groups.
\end{Rem}
\noindent In addition to a presentation for $K(W_k(q,N_1))$, we also
get a more explicit presentation of elliptic curves $E(\f_{q^k})$ in certain cases.

\begin{Thm} \label{ellcoker}
If $E(\f_q) \cong \z / N_1\z$, as opposed to the product of two
cyclic groups, and $End(E) \cong \z[\pi]$, then
$$E(\f_q^k) \cong \z^k \bigg / M_k\z^k$$ for all $k\geq 1$.  That
is, $E(\f_{q^k})$ is the cokernel of the image of $M_k$.
Furthermore, there exists a point $P \in E(\f_{q^k})$ with property
$\pi^m(P) \not = P $ for all $1 < m < k$ such that we can take
$\z^k$ as being generated by $\{P,\pi(P),\dots, \pi^{k-1}(P)\}$
under this presentation.
\end{Thm}

\begin{proof}
A theorem of Lenstra \cite{Len} says that an \emph{ordinary} elliptic
curve over $\f_q$ has a group structure in terms of its endomorphism
ring, namely, $$E(\f_{q^k}) \cong End(E) \bigg / (\pi^k -1).$$
Wittman \cite{Wit} gives an explicit description of the
possibilities for $End(E)$, given $q$ and $E(\f_q)$.  It is well
known, e.g. \cite{Silver}, that the endomorphism ring in the
ordinary case is an order in an imaginary quadratic field. This
means that $$End(E) \cong  \mathcal{O}_g = \z \oplus g \delta\z$$
for some $g \in \z_{\geq 0}$ and $\delta = \sqrt{D}$ or
$\frac{1+\sqrt{D}}{2}$ according to $D$'s residue modulo $4$. Wittman
shows that for a curve $E$ with conductor $f$, the possible $g$'s
that occur satisfy $g |f$ as well as $$n_1 = \gcd(a-1,g/f).$$  The
conductor $f$ and constant $a$ are computed by rewriting the
Frobenius map as $\pi = a + f\delta$, and $n_1$ is the unique
positive integer such that $E(\f_q) \cong \z / n_1\z \times \z /
n_2\z  ~~ (n_1 | n_2).$

We focus here on the case when $g=f$ and $End(E) \cong \z[\pi].$  In
particular, $n_1$ must be equal to one in this case, and so the
condition that $End(E)=\z[\pi]$ is actually a sufficient hypothesis.
Since $E(\f_{q^k}) \cong \z[\pi]/(1-\pi^k)$ in this case, we get
$$E(\f_{q^k}) \cong \z[x] / (x^2 - (1+q-N_1)x+q,~~x^k-1)$$ with $x$
transcendent over $\q$.
Thus \begin{eqnarray*}&&E(\f_{q^k}) \cong \z\{1,x,x^2,\dots, x^{k-1}\}\bigg/ \\
&&\hspace{1em}\bigg(x^2-(1+q-N_1)x+q,~ x^3-(1+q-N_1)x^2+qx,~\dots, ~x^{k-1}-(1+q-N_1)x^{k-2}+qx^{k-3}, \\
&&\hspace{2em}
~1-(1+q-N_1)x^{k-1}+qx^{k-2},~x-(1+q-N_1)+qx^{k-1}\bigg)
\end{eqnarray*}
and using matrix $M_k$, as defined above, we obtain the desired
presentation for $E(\f_{q^k})$ in this case.
\end{proof}

\begin{Ques}
What can we say in the case of another endomorphism ring, or the
case when $E(\f_q)$ is not cyclic?
\end{Ques}

\begin{Ques} Even more generally, are there other families of varieties $\mathcal{F}$ and other families of graphs $\mathcal{G}$ such that the Jacobian groups of $\mathcal{F}$ correspond to the critical groups of $\mathcal{G}$?
\end{Ques}

\section{Connections to Deterministic Finite Automata}

We conclude this paper with connection to yet a third field of mathematics.  A deterministic finite automaton (DFA) is a finite state machine $M$
built to recognize a given language $L$, i.e. a set of words in a
specific alphabet.  To test whether a given word $\omega$ is in
language $L$ we write down $\omega$ on a strip of tape and feed it
into $M$ one letter at a time.  Depending on which state the machine
is in, it will either accept or reject the character.  If the
character is accepted, then the machine's next state is determined
by the previous state and the relevant character on the strip. As
the machine changes states accordingly, and the entire word is fed
into the machine, if all letters of $\omega$ are accepted, then
$\omega$ is an element of language $L$.

For our purposes we consider an automaton $M_G$ with three states, which we label as $A, B$, and  $C$.  In state $A$ we either
accept a character in $\{1+q,\dots, q+t\}$ and return to state $A$, accept a character in $\{1,\dots, q\}$ and move to state $B$, or
accept the character $0$ and move to state $C$.

On the other hand, in state $B$ we either accept a character in
$\{1+q,\dots, q+t\}$ and move to state $A$, accept a character in
$\{1,\dots, q\}$ and return to state $B$, or accept character $0$ and
move to state $C$.

Finally, in state $C$ we either accept a character in $\{1+q,\dots,
q+t\}$ and move to state $A$, or accept character $q$ and return to
state $C$.  A character in $\{1,\dots, q\}$ is not accepted while in
state $C$.  This DFA is illustrated here, with its transition matrix
also given.

\begin{figure} [hpbt]
\includegraphics[width=3.125in,height=2.5in]{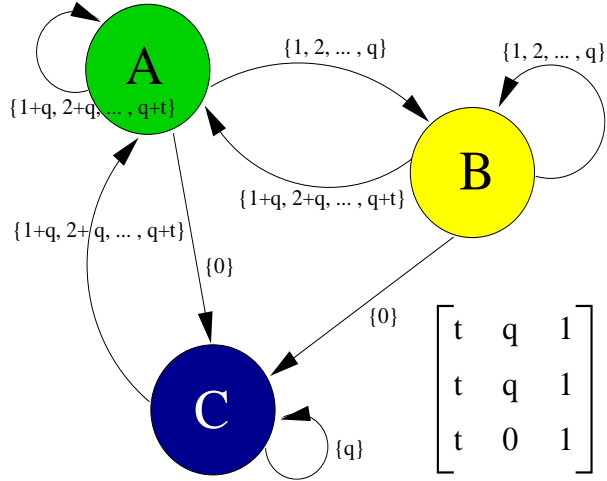}
\caption{Deterministic finite automaton $M_G$.}\end{figure}

We consider the set of words $\mathcal{L}(q,t)$ which are accepted
by $M_G$ with the properties (1) the initial state of $M_G$ is the
same as its final state, and (2) $M_G$ is in state $A$ at some point
while verifying $\omega$.  Comparing definitions, we observe that
the set of such words is in fact the set of critical configurations,
as described in Section $4$.  We can in fact characterize
this set even more concretely.

\begin{Prop} \label{QTRegLang} The set $\mathcal{L}(q,t)$ is a \emph{regular language}, i.e. a
set of words which can be described by a DFA $\mathcal{D_L}.$  In
particular, word $\omega$ is in $\mathcal{L}(q,t)$ if and only if
$\omega$ is admissible by $\mathcal{D_L}$.
\end{Prop}

\begin{proof}
Regular languages can be built by taking complements, the Kleene
star, unions, intersections, images under homomorphisms, and
concatenations.  Thus we can prove $\mathcal{L}(q,t)$ is regular by
decomposing it as the union over all cyclic shifts, a homomorphism,
of concatenation of the blocks of form $B,M_1,M_2,\dots, M_k$.
\end{proof}

More explicitly, we can also use $M_G$ to build a $DFA$ recognizing
$\mathcal{L}(q,t)$, thus giving a second proof. First, machine $M_G$
as described is not technically a DFA since we are not specifying
which of the three states is the initial state and what state the
DFA moves to from state $C$ when it encounters a character in
$\{0,1,2,\dots, q-1\}$. We also have the added restrictions that a
word is only admissible if the DFA goes through state $A$ along its
path, and that words admitted by closed paths in this DFA.

However, this can be easily rectified.  First, we add four
additional states: an initial state $I$, two states $\tilde{B}$
$\tilde{C}$, and a dead state $D$.  Start state $I$ connects to
states $A$, $\tilde{B}$ and $\tilde{C}$, moving to $A$ if the first
letter is $\geq 1+q$, moving to $\tilde{C}$ if the first letter is
$0$, and moving to $\tilde{B}$ otherwise.  Additionally, state
$\tilde{B}$ connects to $A$, $\tilde{B}$, and $\tilde{C}$ just as
$B$ connects to $A$, $B$, and $C$; similarly, $\tilde{C}$ connects
to $A$ and $\tilde{C}$ just as $C$ connects to $A$ and $C$. When the
machine is in state $C$ or $\tilde{C}$, and a character from
$\{0,1,2,\dots, q-1\}$ is read, the machine moves to the dead state
$D$ which always loops back to itself. Letting states $A$, $B$, and
$C$ be the only final/terminal states of this DFA, we now have the
property that a word is only admissible if the DFA goes through
state $A$ at some point along its path.

We now have to deal with the restriction that a word is admissible
only if the word induces a cycle of states in the DFA.  To this end,
we expand the DFA even further essentially copying it three times
and making sure the terminal states correspond to the first state
reached, i.e. immediately following the start state.

\subsection{Another Kind of Zeta Function}

Returning to the original formulation, critical configurations
correspond to closed paths in DFA $M_G$ which go through state $A$.
Since a cycle involving both states $B$ and $C$ but not state $A$ is
impossible, the only cycles we need to disallow are those containing
only state $B$ and those cycles containing only state $C$.  Such
words, i.e. the set $\mathcal{L}(q,t)$ is a \emph{cyclic language}
since the set is closed under circular shift (more precisely $uv \in
\mathcal{L}(q,t)$ if and only if $vu \in \mathcal{L}(q,t)$ for all
$u$, $v$).

Regular cyclic languages such as $\mathcal{L}(q,t)$ were studied in
\cite{BerReut}, and we can even define a zeta function for them.
The zeta function of a cyclic language $L$ is defined as
$$\zeta(L,T) = \exp\bigg( \sum_{k=1}^{\infty} \mathcal{W}_k \frac{T^k}{
k}\bigg)$$ where $\mathcal{W}_k$ is the number of words of length
$k$.  Alternatively, this can be written as
\begin{eqnarray*}\zeta(L,T) &=&
\exp\bigg(\sum_{\mathrm{allowed~closed~paths~}P}
(\mathrm{\#~words~admissible~by~path~}P)~~T^k\bigg).\end{eqnarray*}

\begin{Thm} [Berstel and Reutenauer]
The zeta function of a cyclic and regular language is rational.
\end{Thm}

\begin{proof}
See  \cite{BerReut} or \cite{Reut1}.
\end{proof}

\noindent This observation motivated Berstel and Reutenauer to aks the following question.

\begin{Ques}
For a given algebraic variety $V$ with zeta function $Z_V(T)$ (also rational by Dwork \cite{Dwork}), can we find a cyclic regular language $L$ such that $\zeta(L,T)=Z_V(T)$?  
\end{Ques}

\noindent We will come back to this question momentarilly.

The \emph{trace} of an automaton $\mathcal{A}$ is the language of
words generated by closed paths in $\mathcal{A}$.  Such a language
is always cyclic and regular by construction, and in fact has a zeta
function with an explicit formula.

\begin{Prop}
$$\zeta(trace(\mathcal{A})) = \frac{ 1}{\det(I-M\cdot T)},$$ where $M$
encodes the number of directed edges between state $i$ and state $j$
in $\mathcal{A}$.
\end{Prop}
This matrix is in fact the transition matrix provided above with the
example of automaton $M_G$.
\begin{proof}
We omit this proof, again referring the reader to \cite{BerReut}.
However, we also take this opportunity to mention that the proof is
an application of MacMahon's Master Theorem \cite{Macmahon} which
relates the generating function of traces to a determinantal
formula, or more precisely the characteristic polynomial of a
matrix.  Moreover, analogies between the zeta function of a language
and the zeta function of a variety are even clearer since the proof
of the Weil conjectures via \'{e}tale cohomology also involve such
determinantal expressions.
\end{proof}

\begin{Thm} \label{ZetaLangQT}
For $\mathcal{L}(q,t)$ as given in Proposition \ref{QTRegLang}, the zeta function   $\zeta(\mathcal{L}(q,t))$  equals  $$\frac{(1-qT)(1-T)}{
1-(1+q+\mathcal{W}_1)T + qT^2}.$$
\end{Thm}

\begin{proof}
Using the above terminology, we can describe the set of critical
configurations of $W_k(q,t)$ as the language obtained by taking
the trace of $M_G$ \emph{minus} the trace of cycles only containing
state $B$ \emph{minus} the trace of cycles only containing state
$C$.  We again note that all other circuits with the same initial
and final state necessarily need to contain state $A$ since there
are no cycles containing both state $B$ and $C$ but not $A$. There
is no way to go from state $C$ to state $B$ without going through
state $A$ first, given the definition of $M_G$.

Thus the zeta function of this cyclic language is given as
$$\frac{\det([1-qT]) \det([1-T])}{\det (I-MT)}$$ where the factor of
$\det([1-qT])$ correspond to the trace of cycles containing state
$B$ alone, and $\det([1-T])$ corresponds to the trace of cycles
containing state $C$ alone.  On the other hand, matrix $M$ is the
$3$-by-$3$ matrix encoded by the number of directed edges between
the various states.
$$\left[\begin{matrix} t & q & 1 \\
t & q & 1 \\
t & 0 & 1\end{matrix}\right]$$
Thus we arrive at the desired expression for
$\zeta(\mathcal{L}(q,t))$, namely $$\exp\bigg( \sum_{k=1}^\infty
\frac{\mathcal{W}_k}{k} T^k\bigg) = \frac{(1-qT)(1-T) }{
1-(1+q+\mathcal{W}_1)T + qT^2}$$ where $\mathcal{W}_k$ equals the
number of primitive cycles in $M_G$, which contain state $A$ but
starting at any of the three states.
\end{proof}

At this point, we have yet another proof of the Theorem
\ref{MainMus}, which states $N_k = -\mathcal{W}_k(q,-N_1)$.  The
reasoning being
\begin{eqnarray*} \exp\bigg(\sum_{k\geq 1}\frac{\mathcal{W}_k}{ k} T^k\bigg) &=& \frac{(1-qT)(1-T) }{ 1 - (1+q+t)T + qT^2} \\
&=& \bigg(\frac{ 1 - (1+q+t)T + qT^2 }{ (1-qT)(1-T)}\bigg)^{-1} \\
&=& (Z(E,T)|_{N_1=-t})^{-1} \\
&=& \exp\bigg(-\sum_{k\geq 1}\frac{N_k }{ k} T^k\bigg)\bigg|_{N_1=-t}.
\end{eqnarray*}
\noindent Additionally, we have answered Berstel and Reutenauer's question for elliptic curves, up to an issue of sign.  Perhaps other cyclic languages, constructed from the critical groups of graphs or otherwise, correspond to other algebraic varieties analogously.

In this paper, we have continued the study of \cite{Mus} which explored the theory of elliptic curves over finite fields with an eye towards combinatorial results.  The relationship between elliptic curves and spanning trees appears even more pronounced than one would have guessed from the motivation of Theorem \ref{MainMus}.
Not only do we have formal identities relating the number of
spanning trees of wheel graphs and number of points on elliptic
curves, but we also have connections between the corresponding group
structures of these two families of objects.  Characterizations of critical groups in terms of combinatorics on words also appears fruitful.  The connections
described here inspire further exploration for connections between
these three topics.

%
\bibliographystyle{amsalpha}

\end{document}